\documentclass[english]{article}
\usepackage[T1]{fontenc}
\usepackage[utf8]{inputenc}
\usepackage{lmodern}
\usepackage{amssymb,authblk}
\usepackage{mathtools}
\usepackage{amsthm}
\usepackage{stmaryrd}
\usepackage{mathrsfs}
\usepackage{amsfonts}
\usepackage{faktor}
\usepackage{xcolor}
\usepackage{tcolorbox}
\usepackage{graphicx}
\usepackage{multicol}
\usepackage{array}
\usepackage[a4paper]{geometry}
\usepackage{shuffle}
\usepackage[english]{babel}

\usepackage{subcaption}
\usepackage{hyperref}
\usepackage{tikz,tikz-cd}
\usepackage{fp}
\usepackage{ifthen}
\usepackage{calc}
\usetikzlibrary{automata, positioning, arrows}

\newcounter{thmcount}
\theoremstyle{definition}
\newtheorem{defi}[thmcount]{Definition}
\newtheorem{Eg}[thmcount]{\textbf{Example}}
\newtheorem{Egs}[thmcount]{\textbf{Examples}}
\newtheorem{Rq}[thmcount]{\textbf{Remark}}
\newtheorem{Not}[thmcount]{\textbf{Notation}}
\newtheorem*{Eg*}{Exemple}
\newtheorem*{defi*}{Definition}
\newtheorem*{Rq*}{\textbf{Remark}}

\newtheorem{innercustomgeneric}{\customgenericname}
\providecommand{\customgenericname}{}
\newcommand{\newcustomtheorem}[2]{%
	\newenvironment{#1}[1]
	{%
		\renewcommand\customgenericname{#2}%
		\renewcommand\theinnercustomgeneric{##1}%
		\innercustomgeneric
	}
	{\endinnercustomgeneric}
}
\newcustomtheorem{customdefi}{Definition}
\newcustomtheorem{customrq}{Remark}
\newcustomtheorem{customEgs}{Examples}
\newcustomtheorem{customEg}{Example}

\theoremstyle{plain}
\newtheorem{Prop}[thmcount]{Proposition}
\newtheorem{Lemme}[thmcount]{Lemma}
\newtheorem{Cor}[thmcount]{Corollary}
\newtheorem{thm}[thmcount]{Theorem}
\newtheorem*{thm*}{Theorem}
\newtheorem*{Cor*}{Corollary}
\newtheorem*{Lemme*}{Lemma}
\newtheorem*{Prop*}{Proposition}

\newtheorem{innercustomgenerictwo}{\customgenericname}
\providecommand{\customgenericname}{}
\newcommand{\newcustomtheoremplain}[2]{%
	\newenvironment{#1}[1]
	{%
		\renewcommand\customgenericname{#2}%
		\renewcommand\theinnercustomgenerictwo{##1}%
		\innercustomgenerictwo
	}
	{\endinnercustomgeneric}
}
\newcustomtheoremplain{customThm}{Theorem}
\newcustomtheoremplain{customLemma}{Lemma}
\newcustomtheoremplain{customProp}{Proposition}
\newcustomtheoremplain{customCor}{Corollary}

\NewDocumentCommand{\Y}{m m}{\raisebox{-0.25\height}{\begin{tikzpicture}[line cap=round,line join=round,>=triangle 45,x=0.25cm,y=0.25cm]
			\draw [line width=.5pt] (0.,0.)-- (0.,1.);
			\draw [line width=.5pt] (0.,1.)-- (-1.,2.);
			\draw [line width=.5pt] (0.,1.)-- (1.,2.);
			\draw[above] (1,2) node {\small{#2}};
			\draw[above] (-1,2) node {\small{#1}}; 
\end{tikzpicture}}}

\NewDocumentCommand{\Yind}{}{\begin{tikzpicture}[line cap=round,line join=round,>=triangle 45,x=0.1cm,y=0.1cm]
			\draw [line width=.5pt] (0.,0.)-- (0.,1.);
			\draw [line width=.5pt] (0.,1.)-- (-1.,2.);
			\draw [line width=.5pt] (0.,1.)-- (1.,2.);
\end{tikzpicture}}

\NewDocumentCommand{\YY}{m m m m}{\raisebox{-0.3\height}{\begin{tikzpicture}[line cap=round,line join=round,>=triangle 45,x=0.3cm,y=0.3cm]
			\draw (0,0)--(0,1);
			\draw (0,1)--(-2,3);
			\draw (-1.25,2.25)--(-0.5,3);
			\draw (1.25,2.25)--(0.5,3);
			\draw (0,1)--(2,3);
			\draw[above] (-2,3) node{\small{#1}} ;
			\draw[above]  (-0.5,3) node{\small{#2}};
			\draw[above]  (0.5,3) node{\small{#3}};
			\draw[above]  (2,3) node{\small{#4}};
\end{tikzpicture}}}

\NewDocumentCommand{\balaisg}{m m m}{\raisebox{-0.3\height}{\begin{tikzpicture}[line cap=round,line join=round,>=triangle 45,x=0.4cm,y=0.4cm]
			\draw (0,0)-- (0,1);
			\draw (0,1)-- (-1,2);
			\draw (-0.5,1.5)-- (0,2);
			\draw (0,1)-- (1,2);
			\draw[above] (-1,2) node {\small{#1}};
			\draw[above] (0,2) node {\small{#2}};
			\draw[above] (1,2) node {\small{#3}};
\end{tikzpicture}}}

\NewDocumentCommand{\balaisd}{m m m}{\raisebox{-0.3\height}{\begin{tikzpicture}[line cap=round,line join=round,>=triangle 45,x=0.4cm,y=0.4cm]
			\draw (0,0)-- (0,1);
			\draw (0,1)-- (-1,2);
			\draw (0.5,1.5)-- (0,2);
			\draw (0,1)-- (1,2);
			\draw[above] (-1,2) node {\small{#1}};
			\draw[above] (0,2) node {\small{#2}};
			\draw[above] (1,2) node {\small{#3}};
\end{tikzpicture}}}

\newcommand{\dbleftbrace}{\left\{\mskip-5mu\left\{}
\newcommand{\dbrightbrace}{\right\}\mskip-5mu\right\}}

\NewDocumentCommand{\multiset}{m}{\dbleftbrace #1 \dbrightbrace}
\NewDocumentCommand{\CQMM}{}{Cartier-Quillen-Milnor-Moore}

\newcommand{\IEM}[2]{\llbracket #1,#2 \rrbracket}
\newcommand{\e}{\varepsilon}
\newcommand{\K}{\ensuremath{\mathbb{K}}}
\newcommand{\N}{\mathbb{N}}

\newcommand{\M}{\mathcal{M}}

\NewDocumentCommand{\U}{m}{\ensuremath{\mathcal{U}\left(#1\right)}}

\NewDocumentCommand{\h}{}{\ensuremath{\mathfrak{h}}}
\NewDocumentCommand{\PH}{}{Post-Hopf}
\NewDocumentCommand{\PL}{}{Post-Lie}
\NewDocumentCommand{\libra}{ m m }{\ensuremath{\left[ #1, #2\right]}}
\NewDocumentCommand{\libraJ}{ m m }{\ensuremath{\left[ #1, #2\right]_J}}
\NewDocumentCommand{\librao}{ m m }{\ensuremath{\left[ #1, #2\right]^{\mathrm{op}}}}
\NewDocumentCommand{\lhdo}{}{\ensuremath{\lhd^{\mathrm{op}}}}
\NewDocumentCommand{\cdoto}{}{\ensuremath{\cdot^{\mathrm{op}}}}
\NewDocumentCommand{\Deltao}{}{\ensuremath{\Delta^{\mathrm{cop}}}}

\DeclareMathOperator{\X}{\mathfrak{X}(\M)}
\DeclareMathOperator{\C}{\mathcal{C}^{\infty}(\M)}

\DeclareMathOperator{\ComTrias}{ComTrias}
\DeclareMathOperator{\End}{End}
\DeclareMathOperator{\Hom}{Hom}
\DeclareMathOperator{\Prim}{Prim}
\DeclareMathOperator{\ima}{Im}

\DeclareMathOperator{\car}{char}

\DeclareMathOperator{\Id}{Id}
\DeclareMathOperator{\supp}{supp}

\DeclareMathOperator{\Leaf}{Leaf}
\DeclareMathOperator{\len}{deg}

\DeclareMathOperator{\Lie}{\mathcal{L}ie}
\DeclareMathOperator{\F}{\mathcal{F}}

\DeclareMathOperator{\Tree}{BT}

\DeclareMathOperator{\Treedecleaf}{BTDL}
\DeclareMathOperator{\Fdecleaf}{FDL}
\DeclareMathOperator{\FLT}{FLPT}

\DeclareMathOperator{\lhdp}{\blacktriangleleft}

\providecommand{\keywords}[1]{\textbf{\textit{Keywords.}} #1}
\providecommand{\AMSclass}[1]{\textbf{\textit{AMS classification.}} #1}

\begin{document}
\title{{\bfseries The \CQMM{} theorem in the \PH{} case}}
\author{Pierre Catoire
	\thanks{Contact: \href{mailto: catoire_research@proton.me; 
		}
	{catoire\_research@proton.me
}}}
\affil{\small{Univeristé de Montpellier, IMAG, Institut Montpelliérain Alexander Grothendieck,
		 Place Eugène Bataillon, Montpellier, 34070, France.
	}}

\maketitle
\begin{abstract}
	\PH{} algebras have appeared in many works involving differential geometry or geometric integration. It can be obtained considering the enveloping algebra of any Post-Lie algebras. With an algebraic point of view, we classify \PH{} algebras with an analogous of \CQMM{} theorem in the cocommutative case. We also give two example of algebras. One is built on sentences taking a free associative product as $\rhd$ and the second on trees with the grafting operator for $\rhd$.
\end{abstract}


\AMSclass{17B060,17A30, 16T30,16T05, 16T10
}


\section*{Introduction}

The notion of \PL{} algebras was introduced at first by B.~Vallette~\cite{Vallette} as the Koszul dual of the operad $\ComTrias$ in order to study the homology of some partition posets. Later on, the \PL{} structure arose from a quite different context: \emph{differential geometry}. Indeed, given a smooth manifold $\M$ satisfying some geometric properties, its space of vector fields has a \PL{} structure. This discovery led to many development in numerical approximation to get better computations in those cases using this structure like in geometric integration~\cite{GNI,FreePL}. This paper being more algebraic than geometric, I will give some clues about this structure later.

For this \PL{} algebra, one can build a \PH{} algebra considering its enveloping algebra. Geometrically, it represents the algebra of differential operators over the manifold $\M$. Thanks to the work of Butcher~\cite{Butcher_63,Butcher_72, Butcher_96, Sanz-Serna}, in numerics, any Runge-Kutta methods has a development as a \emph{Butcher series} representing the numerical scheme. A Butcher series is a representation of some family of numerical schemes (including Runge-Kutta methods) using a formal series whose variables are trees (or words). Each tree represents an elementary differential operators of some orders involved in the numeric schemes. In our particular geometric case, it turns out that to perform the usual computations with Butcher series, the good algebraic setting is the one of the Munthe-Kaas--Wright Hopf algebra obtained as the enveloping algebra of the free \PL{} algebras described by Lyndon words over trees endowed with the left grafting operator.

In this paper, we will mainly be interested in the algebraic structure of \PH{} algebras in order to understand better its behaviour. With an algebraic point of view on \PH{} algebras, we manage to see recover the classic \CQMM{} theorem involving the additional operation of \PL{} algebras. We also give two examples of such algebras one built on sentences and the other one on trees which turns out to be isomorphic to the Munthe-Kaas--Wright algebra. We also state some algebraic properties 

More general structures exist called \emph{Yetter-Drinfeld} Hopf algebra~\cite{Sciandra_25}. It is a generalization of the \PH{} structure as assuming cocommutativity we recover the definition of our \PH{} algebras. In the context of Yetter-Drinfeld algebras, non-cocommutative examples exist whereas for \PH{} algebras there is known examples.

Before beginning our algebraic study, let us detail the context in which those algebraic structure appear in differential geometry~\cite{LPostLie,FreePL}.

Let $\M{}$ be a smooth manifold and let us denote $\X$ the space of \emph{vector fields} over $\M$. We see a \emph{vector field} $X$ as a map $\C \rightarrow \C$ such that for any $f,g\in\C$:
\[
X(f\cdot g)= X(f)\cdot g + f \cdot X(g).
\] 
Without any geometric assumption, $\X$ has a $\Lie{}$ algebra structure with the so called \emph{Jacobi bracket}~\cite{Nomizu_Kobayashi} defining a new element of $\X$ for any vector fields $X,Y\in\X$ by ${\libraJ{X}{Y}=X(Y(f))-Y(X(f))}$. 
Moreover, one can add an \emph{affine connection} $\rhd$ which is a linear map ${\X\otimes \X \rightarrow \X}$ satisfying for any vector fields $X,Y$ and $f\in\C$:
\[
f(X\rhd Y)= (f\cdot X)\rhd Y \text{ and } X \rhd (f\cdot Y) = X(f)Y+ f(X\rhd Y).
\] 
From those data, one can define the \emph{torsion} $T:\X\otimes \X\rightarrow \X$ and the \emph{curvature} $R:\X\times \X \rightarrow \End(\X)$ of the connection $\rhd$ for any vector fields $X,Y$ and $Z$ by:
\[
 T(X,Y)=X\rhd Y - Y \rhd X - \libraJ{X}{Y} \text{ and } R(X,Y)(Z)= X\rhd (Y \rhd Z) - (X \rhd Y)\rhd Z - \libraJ{X}{Y}\rhd Z.	
\]
Hence, we can express the curvature with the torsion:
\[
R(X,Y)(Z)= X \rhd (Y\rhd Z) - (X \rhd Y)\rhd Z + (Y \rhd X) \rhd Z - (X \rhd Y)\rhd Z - T(X,Y)\rhd Z.
\]
Those two maps are related with the following identity called the \emph{first Bianchi's identity} for any $X,Y,Z\in\X$:
\begin{equation}\label{eq:Bianchi}
	\sum_{i=0}^2 \left(T(T(\sigma^i(X),\sigma^i(Y)),\sigma^i(Z)) + \sigma^i(X)\rhd T(\sigma^i(Y),\sigma^i(Z))\right)=\sum_{i=0}^2 R(\sigma^i(X),\sigma^i(Y))(\sigma^i(Z))
\end{equation}

where $\sigma: X\mapsto Y, Y\mapsto Z, Z \mapsto X.$
Moreover, let us remind the definition of the \emph{covariant derivative} of $T$, denoted $\nabla T: \X\rightarrow \End(\X^2,\X)$, for any $X\in\X$ by:
\begin{equation}\label{eq:coderiv}
	\forall\,Y,Z\in\X, (\nabla T)(X)(Y,Z)=X\rhd T(Y,Z) - T(X \rhd Y,Z) - T(Y, X \rhd Z).
\end{equation}

Considering $T$ as new Lie bracket, $(\X,T,\rhd)$ is a \PL{} algebra under the assumption that the curvature $R$ of the connection $\rhd$ is $0$ (the connection is then called \emph{flat}) and its torsion $T$ satisfies $\nabla T=0$ (then the connection is said to have \emph{constant torsion}). Indeed, equation~\eqref{eq:Bianchi} becomes the Jacobi identity for the torsion Lie bracket, equation~\eqref{eq:coderiv} give rise to the left development of $\rhd$ onto a Lie bracket and finally the definition of the curvature with $R=0$ gives the last property needed.
Theses structures are appearing in numerical geometric integration. We refer the reader to specialized papers for more details~\cite{Busnot_Laurent_2025,Floystad_2020,Curry2018,Al_Kaabi_2022}. 

\subsection*{Document structure}

\begin{enumerate}
	\item In the first section, we give basic definitions about right \PH{} algebras and the construction of \PH{} structures over the tensorial algebra~\cite{siso}.
	\item The second section introduces the main theorem of this work implying the \CQMM{} theorem in the \PH{} case. We give the proof of a classic version and a graded version using intermediate lemmas.
	\item The last section gives two examples of \PH{} algebras and applications of this theorem. Hence, we find another combinatorial description of the Munthe-Kaas--Wright algebra~\cite{FreePL,Munthe_Kaas_2006}.  
\end{enumerate}

\subsection*{Acknowledgments}

 I especially thank Loïc Foissy and  Dominique Manchon for reading previous versions of this work and giving piece of advice to improve it.

\subsection*{Funding}

This project was founded by the grant ANR-20-CE40-0007 \emph{Combinatoire Algébrique, Résurgence, Probabilités Libres et Opérades.}

\subsection*{Global notations}

In our work, the base field $\K$ is a commutative field of characteristic zero. Any vector space or algebra will be taken over $\K$.
We list some notations that are used in this paper. The remaining ones will be given when needed.
 Consider $V$ a vector space, $\h$ a Lie algebra and a Hopf algebra $(H,m,1_H,\Delta,\e)$:
\begin{itemize}
	\item we denote $(\Hom(V),\circ)$ the associative algebra of linear maps from $V$ to $V$ with the composition.
	\item we denote by $\tau:V\otimes V\rightarrow V\otimes V, x\otimes y \mapsto y \otimes x$ the flip map. Moreover, we define for appropriate maps $f$ and $g$:
	\[
	f^{\mathrm{op}}\coloneqq f\circ\tau \text{ and } g^{\mathrm{cop}}\coloneqq \tau\circ g.
	\]
	\item we denote $H_+\coloneqq\ker(\e)$ and call it the \emph{augmentation ideal} of $H$.
	\item we denote $\tilde{\Delta}$ the reduced coproduct defined for any $x\in H$ by $\tilde{\Delta}(x)=\Delta(x)-1\otimes x-x\otimes 1.$
	An element $x\in H$ is said \emph{primitive} if $\tilde{\Delta}(x)=0.$
	Moreover, we use the Sweedler notation for both coproducts 
	\[
	\Delta(x)\coloneqq x^{(1)}\otimes x^{(2)} \text{ and } \tilde{\Delta}(x)\coloneqq x'\otimes x''.
	\]
	\item for any $k\in\N^*$, we define $\Delta^{(0)}(x)= x, \tilde{\Delta}^{(0)}= x-\e(x)1_H$ and by induction:
	\[
	\Delta^{(k)}=\left(\Delta \otimes \Id^{\otimes k-1}\right)\circ \Delta^{(k-1)} \text{ and } \tilde{\Delta}^{(k)}=\left(\tilde{\Delta} \otimes \Id^{\otimes k-1}\right)\circ \tilde{\Delta}^{(k-1)}.
	\]
	\item we put $V^{\otimes 0}=\K\cdot 1$ where $1$ is the empty word.
	\item we put $T(V)=\bigoplus_{n\in\N} V^{\otimes n}$ the tensor algebra. It is a Hopf algebra endowed with the concatenation product and the deshuffle coproduct defined for any $v\in V$ by $\Delta(v)=1\otimes v + v\otimes 1$.
	\item given a linear map $\ltimes:V\otimes V\rightarrow V$, we denote $a_\ltimes$ its associator, defined for any $x,y,z\in V$ by $a_\ltimes(x,y,z)=(x\ltimes y)\ltimes z-x\ltimes (y\ltimes z).$
	\item $\U{\h}$ is the enveloping algebra of $\h$ equal to $\faktor{T(\h)}{I}$ where $I$ is the ideal of $T(\h)$ generated by $\left\{ xy-yx-\libra{x}{y} \,\middle| \, x,y\in\h \right\}$.
	It inherits the Hopf algebra structure from $T(\h)$. 
	\item let $(C,\Delta)$ be a coalgebra and $(A,m)$ be an algebra. We consider $*$ the convolution product of $\Hom(C,A)$ defined for any $(f,g)\in\Hom(C,A)^2$ by $f*g = m\circ (f\otimes g) \circ \Delta.$
	\item let $(A,+,\times)$ be an associative algebra. Let $B$ be a subspace of $A$ and $W$ a subspace of $V$. We denote by $\left\langle W \right\rangle$ the subspace of $V$ generated by $W$ and $\left\langle B \right\rangle$ is the sub-algebra of $A$ generated by $B$.
\end{itemize}

\section{Left and Right Post-Lie and Post-Hopf algebras}

\subsection{Definitions}
 
We first introduce the notions of \PL{} and \PH{} algebras. 
We define right \PL{} algebras~\cite{siso}:
\begin{defi}[right \PL{} algebras]\label{defi:RPL}
	A \emph{right \PL{} algebra} is a triple $(\h, [,], \lhd)$ where $(\h,[,])$ is a Lie algebra and $\lhd:\h\otimes\h \rightarrow \h$ is a linear map satisfying for all $x,y,z\in\h$:
	\begin{align}
		x\lhd \libra{y}{z}&=a_{\lhd}(x,y,z)-a_{\lhd}(x,z,y) \label{RPL2} \\
							&=(x\lhd y)\lhd z -x\lhd (y\lhd z)-(x\lhd z)\lhd y + x\lhd (z\lhd y),\nonumber \\
		\libra{x}{y}\lhd z&=\libra{x\lhd z}{y}+\libra{x}{y\lhd z}. \label{RPL1}
	\end{align}
If the Lie algebra is abelian, then we find back the definition of \emph{right Pre-Lie algebras.}
\end{defi}


\begin{Rq}
It turns out that a \PL{} algebra has a second Lie structure with the Lie bracket $\{,\}$ defined for all $x,y\in\h$ by $\{x,y\}:=\libra{x}{y}+x\lhd y-y\lhd x$.
\end{Rq}

In the literature~\cite{EnvPost,LPostLie}, a notion of left \PL{} algebra exists related to the right one by
\begin{Lemme}\label{Lem:linkLR} 
	Let $\left(\h,\libra{}{},\lhd\right)$ be a right \PL{} algebra. Then $\left(\h,\librao{}{},\lhdo\right)$ is a right \PL{} algebra.
\end{Lemme}


Considering the enveloping algebra of any \PL{} algebra $\h$, one can extract the properties satisfied by $\lhd$ in its enveloping algebra $\U{\h}$. Those relations give rise to the definition of right \PH{} algebras (see Foissy's work~\cite[lemma~1 and proposition~1]{siso}):
\begin{defi}[right \PH{} algebras] \label{defi:RPH}
	A \emph{right \PH{} algebra} is a 6-tuple $(H,\cdot,1,\Delta,\e,\lhd)$ such that $(H,\cdot,1,\Delta,\e)$ is a Hopf algebra and $\lhd:H\otimes H \rightarrow H$ is a coalgebra morphism satisfying the following properties for all $x,y,z\in H$:
	\begin{align}
		(x\cdot y) \lhd z&=\left( x \lhd z^{(1)}\right)\cdot \left(y \lhd z^{(2)}\right), \label{RPH1} \\
		(x\lhd y)\lhd z&=x\lhd\left( \left(y\lhd z^{(1)}\right)\cdot z^{(2)} \right). \label{RPH2}
	\end{align}
	and such that the right multiplication $\gamma_{\lhd}:H\rightarrow \Hom(H)$ defined for every $x\in H$ by:
	\[
	\gamma_{\lhd}(x):\left\lbrace \begin{array}{rcl}
		H & \rightarrow & H \\
		y & \mapsto & y\lhd x,
	\end{array}\right.
	\]
	is invertible in the convolution algebra $\Hom(H,\Hom(H))$. 
	So there exists ${\beta\in\Hom(H,\Hom(H))}$ such that for all $x\in H:$
	\[
	\gamma_{\lhd}\left(x^{(1)}\right)\circ \beta\left(x^{(2)}\right)=\beta\left(x^{(1)}\right)\circ\gamma_{\lhd}\left(x^{(2)}\right)=\e(x)\Id_H.
	\]
	When $(H,\cdot)$ is commutative, we will call $(H,\cdot,1,\Delta,\e,\lhd)$ a \emph{right Pre-Hopf algebra}.
\end{defi}
\begin{Rq}
	At first a similar notion appeared under the name of $D$-algebras~\cite{FreePL} from geometric integration, in which the invertibility condition on $\gamma$ was not required.
\end{Rq}

We also define a notion of morphism for \PH{} algebras
\begin{defi}[Post-Hopf morphism]
	A \emph{morphism} of right \PH{} algebras from $(H,\lhd)$ to $(H',\lhdp)$ is a Hopf algebra morphism $g:H\rightarrow H'$ such that for all $x,y\in H$:
	\[
	g(x\lhd y)=g(x)\lhdp g(y).
	\]
\end{defi}

As in the \PL{} case, a notion of left \PH{} algebra arises from the literature~\cite[theorem 2.7]{LPostLie} equivalent to the right one under cocommutativity.
\begin{Prop}\label{Prop:travelLR}
	Let $(H,\cdot, 1, \Delta, \e,\lhd)$ be a right \PH{} algebra such that $\Delta$ is \emph{cocommutative}. Then $(H,\cdoto, 1, \Delta, \e,\lhdo)$ is a left \PH{} algebra.
\end{Prop}
\begin{Rq}
	The requirement on $\Delta$ to be cocommutative is needed such that from the hypothesis that $\gamma_{\lhd}$ is invertible ,we can deduce that the left multiplication $\alpha_{\lhdo}: x \rightarrow (y \mapsto x \rhd y)$ has an inverse in the convolution algebra.
\end{Rq}

We also remind the useful known results in the case of right structures~\cite{LPostLie}
\begin{Lemme}\label{lem:unit}
	Let $(H,\lhd)$ be a right \PH{} algebra. Then, for all $x\in H$ we have:
	\begin{equation*}
		x\lhd 1=x \text{\normalfont{ and }}
		1 \lhd x=\e(x)1.
	\end{equation*}
\end{Lemme}
\begin{thm}\label{thm:PrimPL}
	Let $(H,\lhd)$ be a right (respectively left) \PH{} algebra. Then, $(\Prim(H),\libra{}{},\lhd)$ is a right (respectively left) \PL{} algebra.
\end{thm}
\begin{proof} 
	We refer to the proof of Li, Sheng and Tang~\cite[therorem~2.7]{LPostLie}
\end{proof}

\subsection{An important construction}

In the literature~\cite{siso,FreePL,EnvPost,LPostLie}, an important construction over enveloping algebras arose as a generalization in the \PL{} case of Guin-Oudom procedure~\cite{OudomGuin} in the Pre-Lie case.  In this paper, we will focus on the right extension of the product~\cite{siso}. For completeness, we remind them below:

\begin{Prop}\label{Prop:extendTg}
	Let $V$ be a vector space and $\lhd$ be a magmatic product on $V$. Then, $*$ can be uniquely extended to a map $\lhd:T(V)^{\otimes 2}\rightarrow T(V)$ such that for all $f,g,h\in T(V)$ and $x,y\in V$:	
	\begin{multicols}{2}
		\begin{itemize}
		\item $\e(f\lhd g)=\e(f)\e(g)$;
		\item $\Delta(f\lhd g)=\Delta(f)\lhd \Delta(g)$;
		\item $f\lhd 1=f$;
		\item $1\lhd f=\e(f)1$;
		\item $f\lhd (gy)=(f\lhd g)\lhd y-f\lhd (g\lhd y)$;
		\item $(fg)\lhd h=\sum \left(f\lhd h^{(1)}\right)\left(g\lhd h^{(2)}\right)$;
		\item $(f\lhd g)\lhd h=\sum f\lhd \left(\left(g\lhd h^{(1)}\right)h^{(2)}\right)$.
	\end{itemize}
	\end{multicols}
 \end{Prop}
\begin{Egs}
	Let $V$ be any vector space and $v_1,v_2,v_3,v_4$ be four elements of $V$. Then:
	\begin{align*}
		v_1\lhd v_2&=v_1\lhd v_2, \\
		(v_1v_2)\lhd v_3&=(v_1\lhd  v_3)v_2 +v_1(v_2\lhd  v_3), \\
		v_1\lhd (v_2v_3)&=(v_1\lhd  v_2)\lhd  v_3-v_1\lhd  (v_2\lhd  v_3), \\
		(v_1v_2v_3)\lhd v_4&=(v_1\lhd  v_4)v_2v_3+v_1(v_2\lhd  v_4)+v_1v_2(v_3\lhd  v_4), \\
		v_1\lhd (v_2v_3v_4)&=((v_1\lhd  v_2)\lhd  v_3)\lhd  v_4-(v_1\lhd  (v_2\lhd  v_3))\lhd  v_4-(v_1\lhd  (v_2\lhd  v_4))\lhd  v_3 \\&+v_1\lhd  ((v_2\lhd  v_4)\lhd  v_3) -(v_1\lhd  v_2)\lhd  (v_3\lhd  v_4)+v_1\lhd  (v_2\lhd  (v_3\lhd  v_4)).
	\end{align*}
\end{Egs}
Moreover in the \PL{} case, all those properties remain in the quotient space $\U{\h}$ when $\h$ is a right \PL{} algebra~\cite[proposition~4]{siso}.
\begin{thm}\label{thm:extendUg}
	Let $(\h,\lhd )$ be a right \PL{} algebra. Its magmatic product can be uniquely extended to $\mathcal{U}(\h)$ such that for all $f,g,h\in\mathcal{U}(\h),y\in\h$:
	\begin{multicols}{2}
		\begin{itemize}
		\item $\e(f\lhd g)=\e(f)\e(g)$;
		\item $\Delta(f\lhd g)=\Delta(f)\lhd \Delta(g)$;
		\item $f\lhd 1=f$;
		\item $1\lhd f=\e(f)1$;
		\item $f\lhd (gy)=(f\lhd g)\lhd y-f\lhd (g\lhd y)$;
		\item $(fg)\lhd h=\sum \left(f\lhd h^{(1)}\right)\left(g\lhd h^{(2)}\right)$;
		\item $(f\lhd g)\lhd h=\sum f\lhd \left(\left(g\lhd h^{(1)}\right)h^{(2)}\right)$.
	\end{itemize}
	\end{multicols}
\end{thm}
\section{\CQMM{} theorem for \PH{} algebras}

In this section we prove the classical Cartier-Quillen-Milnor-Moore (CQMM) theorem~\cite{CartierPatras} in the context of \PH{} and \PL{} algebras. For the sake of simplicity, we will restrict ourselves to the right cases. Indeed, they are analogous under the hypothesis of cocommutativity thanks to proposition~\ref{Prop:travelLR}. 

\subsection{Preliminary lemmas and definitions}

We follow a standard sketch of proof adapted to the case of \PH{} algebras.
\begin{defi}[\cite{Montgomery,Sweedler}~coradical filtration]
	Let $(C,\Delta,\e)$ be a \emph{connected coalgebra}, this means for any $x\in C$ there exists $n_x\in\N$ such that $\tilde{\Delta}^{(n_x)}(x)=0.$
	We define the \emph{coradical filtration} of $C$ as the increasing sequence $\left(C^{\leq n}\right)_{n\in\N}$ of subspaces of $C$ defined for all $n\in\N$ by:
	\begin{align*}
		C^{\leq n}\coloneqq\left\lbrace x\in C \,\middle|\, \tilde{\Delta}^{(n)}(x)=0  \right\rbrace
		=\left\lbrace x\in C \,\middle|\, \exists m\leq n,\tilde{\Delta}^{(m)}(x)=0  \right\rbrace.
	\end{align*}
	Given $x\in C,$ we define $\deg(x)= \min \left\{n\in\N \,\middle|\, x\in C^{\leq n}\right\}.$
\end{defi}

Using the coassociativity of $\Delta$, we have this easy lemma:
\begin{Lemme}\label{lem:coalgdecomp}
	Let $(C,\Delta,\e)$ be a connected coalgebra. Then, for all $n\in\N$:
	\begin{equation*} 
		C=C^{\leq 0}\oplus \ker(\e) \text{\normalfont{ and }}  \Delta\left(C^{\leq n}\right)\subseteq \sum_{k=0}^n C^{\leq k}\otimes C^{\leq n-k}.
	\end{equation*}
\end{Lemme}
\begin{Rq}\label{Rq:filtracompat}
	Let $H$ be a bialgebra, connected as a coalgebra, and consider $\left(H^{\leq n}\right)_{n\in \N}$ its coradical filtration. Then, for all $m,n\in\N$:
	\begin{align}\label{eq:connected_coalgebra}
		H^{\leq 0}=(1), &&m\left(H^{\leq n}\otimes H^{\leq m}\right)\subseteq H^{\leq n+m} \text{ and } && \Delta\left(H^{\leq n}\right)\subseteq \sum_{m=0}^n H^{\leq n-m}\otimes H^{\leq m}. 
	\end{align}
\end{Rq}

\begin{Prop}[Filtration over \U{\h}]
	 We put $\U{\h}^{\leq 0}:=\langle 1 \rangle.$
	For any $n\in\N\setminus\{0\},$ we define:
	\[
	\U{\h}^{\leq n}=\bigl\langle \left\{x_1\dots x_k \,\middle|\, k\leq n, \forall 
	i\in\IEM{1}{k},x_i\in \h \right\}\bigr\rangle. 
	\] 
	Then, $\left(\U{\h}^{\leq n}\right)_{n\in\N}$ is the coradical filtration of $\U{\h}$ satisfying equation~\eqref{eq:connected_coalgebra}.
\end{Prop} 
 Using the construction of the antipode in graded bialgebras, we improve proposition~\ref{Prop:travelLR}
\begin{Prop}\label{Prop:graded=PH}
	Let $H$ be a bialgebra connected as a coalgebra, endowed with a coalgebra morphism $\lhd:H^{\otimes 2}\rightarrow H$ such that \eqref{RPH1} and \eqref{RPH2} hold. Then, $(H,\lhd)$ is a \PH{} algebra.
\end{Prop}
\begin{proof}
	The Hopf algebra $H$ is connected as a coalgebra. So $H^{\leq 0}$ is one dimensional.
	One only needs to show that the right multiplication $\gamma_{\lhd}$ is invertible in the convolution algebra $\bigl(\Hom(H,\Hom(H)),*\bigr)$ to end the proof. The remaining properties are easy~\cite{LPostLie}. We build its left inverse $\beta$ and its right inverse $\alpha$ by induction over the degree of the coradical filtration $n$.
	\begin{description}
		\item[Initialization:] we put $\beta(1)=\Id_H=\alpha(1)$. So $\beta*\gamma_{\lhd}(1)=\beta(1)\circ \gamma_{\lhd}(1).$
		By lemma~\ref{lem:unit}, $\gamma_{\lhd}(1)=\Id_H$ as $H^{\leq 0}=\langle1\rangle$. therefore $\beta*\gamma_{\lhd}(1)=\Id_H$  and  $\gamma_{\lhd}*\alpha(1)=\Id_H.$  
		\item[Heredity:] suppose there exists an integer $n$ such that we have defined $\beta$ and $\alpha$ on $H^{\leq n}$ and for all $x\in H^{\leq n}, \beta*\gamma_{\lhd}(x)=\e(x)\Id_H=\gamma_{\lhd}*\alpha(x)$. Let $x\in H^{\leq n+1}$. By lemma \ref{lem:coalgdecomp}, using linearity, we can suppose $\e(x)=0$. 
		Then, we put:
		\[
		\beta(x)\coloneqq -\gamma_{\lhd}(x)-\beta(x')\circ \gamma_{\lhd}(x'') \text{ and } \alpha(x) \coloneqq -\gamma_{\lhd}(x)-\gamma_{\lhd}(x')\circ \alpha(x'').
		\]
		As $\tilde{\Delta}\left(H^{\leq n+1}\right)\subseteq \sum_{k=1}^{n} H^{\leq k}\otimes H^{\leq n+1-k}, \beta(x)$ and $\alpha(x)$ are well-defined, we get:
		\begin{align*}
			\beta*\gamma_{\lhd}(x)&=\beta(x)\circ\gamma_{\lhd}(1)+\beta(1)\circ\gamma_{\lhd}(x)+\beta(x')\circ\gamma_{\lhd}(x'') \\
			&=\beta(x)+\gamma_{\lhd}(x)+\beta(x')\circ\gamma_{\lhd}(x'') \\
			&=0 =\e(x)\Id_H, \\
			\gamma_{\lhd}*\alpha(x)&=\gamma_{\lhd}(x)\circ \alpha(1)+\gamma_{\lhd}(1)\circ \alpha(x)+\gamma_{\lhd}(x')\circ\alpha(x'') \\
			&=\gamma_{\lhd}(x)+\alpha(x)+\gamma_{\lhd}(x')\circ \alpha(x'') \\
			&=0=\e(x)\Id_H.
		\end{align*}
	\end{description} 
Therefore by the induction principle, $\gamma_{\lhd}$ has a left and right inverse. finally as the convolution algebra is associative, it implies the left and right inverse are equal. So $\gamma_{\lhd}$ is invertible. 
\end{proof}

Thanks to proposition~\ref{Prop:graded=PH}, proposition~\ref{Prop:travelLR} and its proof, we get 
\begin{Prop}\label{Prop:OptiLRPH}
	Let $(H,\cdot,1,\Delta,\e,\lhd)$ be a right \PH{} algebra and suppose that $\Delta$ is cocommutative or $H$ is connected as a coalgebra. Then, $\left(H,\cdoto,1,\Deltao,\e,\lhdo\right)$ is a left \PH{} algebra.
\end{Prop}
\begin{proof}
	The cocommutative case is proposition~\ref{Prop:travelLR}. Suppose now $H$ is connected as a coalgebra. Then, $\lhdo$ is a morphism for $\Deltao$. Moreover, $(H,\cdoto,1,\Deltao,\e)$ is a Hopf algebra. Let $x,y,z\in H$ and let us check analogous relations~\eqref{RPH1} and~\eqref{RPH2} for left\PH{} algebras:
	\begin{align*}
		x\lhdo (y\cdoto z)&=(z\cdot y)\lhd x =\left(z\lhd x^{(1)}\right)\cdot \left(y \lhd x^{(2)}\right)
		=\left( x^{(2)} \lhdo y \right)\cdoto \left(x^{(1)} \lhdo z\right), \\
	x \lhdo \left( y\lhdo z \right)&= (z\lhd y) \lhd x = z \lhd \left( \left( y \lhd x^{(1)}\right)\cdot x^{(2)}\right) =\left( x^{(2)}\cdoto \left(x^{(1)}\lhdo y\right)\right)\lhdo z.
	\end{align*}
	Hence, as we have considered the coopposite product, we succeeded. Finally, applying proposition~\ref{Prop:graded=PH} ends the proof. 
\end{proof}

\begin{Lemme}\label{lem:filt+rhd}
	Let $(H,m,1,\Delta,\e,\lhd)$ be a connected (as a coalgebra) \PH{} algebra. 

	Then, for any $n\in\N,	H^{\leq n}\lhd H \subseteq H^{\leq n}.$
\end{Lemme}
\begin{proof}
	Note that for all $x,y\in H_+$, thanks to lemma \ref{lem:unit} we have:
	\begin{align*}
		\tilde{\Delta}(x\lhd y)&=\Delta(x\lhd y)-1\otimes x\lhd y-x\lhd y\otimes 1 \\
		&=\left(1\otimes x+x\otimes 1+\tilde{\Delta}(x)\right)\lhd\left(1\otimes y+y\otimes 1 +\tilde{\Delta}(y) \right) -1\otimes x\lhd y-x\lhd y\otimes 1\\
		&=\tilde{\Delta}(x)\lhd\tilde{\Delta}(y)+\tilde{\Delta}(x)\lhd (1\otimes y+y\otimes 1) \\
		&=\tilde{\Delta}(x)\lhd \Delta(y).
	\end{align*}
Then, let us prove the lemma by induction over $n$.  
\begin{description}
	\item[Initialization:]for $n=0$, the result is obvious by lemma \ref{lem:unit}.
	\item [Heredity:] suppose there exists $n\geq 0$ such that for any $k\leq n, H^{\leq k}\lhd H\subseteq H^{\leq k}$.
	Let $x\in H^{\leq n+1}$ and $y\in H$. If $y\in H^{\leq 0}$, then $x\lhd y=x\in H^{\leq n+1}$. If $y\in H^+$, then by lemma~\ref{lem:unit}, we can assume $x\in H^+$, so:
	\begin{equation*}
		\tilde{\Delta}^{(n+1)}(x\lhd y)=\left( \tilde{\Delta}^{(n)}\otimes \Id\right)\circ \tilde{\Delta}(x\lhd y) =\left( \tilde{\Delta}^{(n)}\otimes \Id\right)\left(\tilde{\Delta}(x)\lhd\Delta(y)\right).
	\end{equation*}
By construction of the filtration, $\displaystyle\tilde{\Delta}(x)\subseteq \sum_{k=1}^{n} H^{\leq k}\otimes H^{\leq n+1-k}.$ Therefore, by the induction hypothesis $\displaystyle\tilde{\Delta}(x)\lhd\Delta(y)\subseteq \sum_{k=1}^{n} H^{\leq k}\otimes H^{\leq n+1-k}$. So $\tilde{\Delta}^{(n+1)}(x\lhd y)=0.$

Hence it proves $H^{\leq n+1}\lhd H\subseteq H^{\leq n+1}$. \qedhere
\end{description}
\end{proof}

From all those lemmas, we deduce:
\begin{Cor}\label{cor:UhPH}
	Let $(\h,\lhd)$ be a \PL{} algebra. Then, $(\U{\h},\lhd)$ is a \PH{} algebra where $\lhd$ extended as in theorem \ref{thm:extendUg}.
\end{Cor} 
\begin{proof}
	One just needs to see that we can apply proposition~\ref{Prop:graded=PH}. Let us remind that $\U{\h}$ has a filtration given by $\left(\U{\h}^{\leq n}\right)_{n\in\N}$. 
	Moreover by theorem~\ref{thm:extendUg}, we see that $\eqref{RPH1}$ and $\eqref{RPH2}$ are also true. Therefore, $\U{\h}$ is a \PH{} algebra. 
\end{proof}

If $\h$ is a Pre-Lie algebra, notice $\U{\h}=S(\h)$. Therefore, we get a stronger version of a result from of Guin and Oudom~\cite[lemma~$2.10$]{OudomGuin} in the context of Pre-Hopf algebras:
\begin{Cor}
	Let $(\h,\lhd)$ be a Pre-Lie algebra. Then, $(S(\h),\lhd)$ is a Pre-Hopf algebra with $\lhd$ extended as in theorem \ref{thm:extendUg}.
\end{Cor}

\subsection{The main theorem and its corollaries}

We introduce our main theorem implying the \CQMM{} theorem: 
\begin{thm}\label{thm:UnivPH}
	Let $(\h,\lhd )$ be a \PL{} algebra, $(A_+,\cdot)$ be a non-unitary associative algebra endowed with a linear map $\lhdp:A_+\otimes A_+\rightarrow A_+$ and let $\phi:\h\rightarrow A_+$ be a linear map such that:
	\begin{equation*}
		\forall x,y\in\h,\phi\left(\libra{x}{y}\right)=\phi(x)\phi(y)-\phi(y)\phi(x) \text{\normalfont{ and }}
		\phi(x\lhd y)=\phi(x)\lhdp \phi(y).
	\end{equation*}
We put $A=\langle 1 \rangle \oplus A_+$ where $1$ is defined as the unit of $A$. We extend $\lhdp$ by linearity to $A$ by:
\begin{equation*}
	\forall x\in A_+, 1 \lhdp x \coloneqq 0, x \lhdp 1 \coloneqq x \text{\normalfont{ and }} 1 \lhdp 1 \coloneqq 1.
\end{equation*}

Moreover, assume for all $x,y\in \left\langle\ima(\phi)\right\rangle,z\in\ima(\phi)$:
\begin{align}
	(x\cdot y)\lhdp z&=(x\lhdp z)\cdot y+x\cdot(y\lhdp z), \label{hyp1} \\
	x \lhdp (y\cdot z)&=(x\lhdp y)\lhdp z-x\lhdp(y\lhdp z). \label{hyp2}
\end{align}
Then, there exists an unique algebra morphism $\Phi:\mathcal{U}(\h)\rightarrow A$ such that $\Phi|_{\h}=\phi$. Moreover, this morphism satisfies for all $f,g,h\in\mathcal{U}(\h)$:
\begin{align}
	\Phi(f\lhd g)&=\Phi(f)\lhdp \Phi(g), \label{PPL0}\\
	\left(\Phi(f)\cdot\Phi(g)\right)\lhdp   \Phi(h) &=\left(\Phi\left(f\right)\lhdp \Phi\left(h^{(1)}\right)\right) \cdot \left( \Phi(g) \lhdp \Phi\left(h^{(2)}\right)\right), \label{PPL1} \\
	\left( \Phi(f)\lhdp  \Phi(g)\right) \lhdp \Phi(h)&= \Phi(f)\lhdp \biggl[ \left( \Phi(g)\lhdp\Phi\left(h^{(1)}\right)\right)\cdot\Phi\left(h^{(2)}\right)\biggr], \label{PPL2}
\end{align}
where $\lhd:\mathcal{U}(\h)^{\otimes 2}\rightarrow \mathcal{U}(\h)$ is the extension defined in theorem~\ref{thm:extendUg}.
\end{thm}
Before proving theorem~\ref{thm:UnivPH}, we need a technical lemma
\begin{Lemme}\label{lem:proofthm}
	Let $n\in\N,n\geq 2$. Let $\Phi:\U{\h}\rightarrow A$ be the algebra morphism given above (satisassumingfying \eqref{hyp1} and \eqref{hyp2}) and assume for any $f,g\in\U{\h}^{\leq n}$:
	\begin{equation}
		\Phi(f\lhd g)=\Phi(f)\lhdp\Phi(g). \label{hyp0}
	\end{equation}
	Then, for all $f,g\in\U{\h}^{\leq n}$ and for all $f_0\in\h$:
	\[
	\Phi(ff_0)\lhdp \Phi(g)=\left(\Phi(f)\lhdp\Phi\left(g^{(1)}\right)\right)\cdot\left(\Phi(f_0)\lhdp \Phi\left(g^{(2)}\right)\right).
	\]
\end{Lemme}
\begin{proof}[Proof of lemma~\ref{lem:proofthm}]
	Before beginning the proof, note for any $f\in\U{\h},\Phi(f)\in\left\langle \ima(\phi) \right\rangle$ because it is a morphism.
	Let $n\in\N^*$ and $f\in\U{\h}^{\leq n}$. We proceed by induction on $r\coloneqq\len(g)$.
	\begin{description}
		\item[Initialization:] we start with $r=0$. Then $g\in \K\cdot 1$, in this case the lemma is true.
		\item[Heredity:]  suppose there exists an integer $0\leq r<n$ such that the conclusion of the lemma is true. Let $g\in \U{\h}$ such that $\len(g)=r+1$. By linearity and the induction hypothesis, we can suppose that $g$ is an element of $\h^{\otimes r+1}$. Then, there exist $g_0\in\h$ and $g'\in\U{\h}^{\leq r}$ such that $g=g'g_0$.  So, using assumption~$\eqref{hyp2}$ and the induction hypothesis we get:
		\begin{align}
			&\Phi(ff_0)\lhdp\Phi\left(g'g_0\right) \nonumber\\
			=&\Phi(ff_0)\lhdp \left(\Phi\left(g'\right)\Phi(g_0)\right) \nonumber\\
			=&\left(\Phi(ff_0)\lhdp \Phi\left(g'\right)\right)\lhdp\Phi(g_0) - \Phi(ff_0)\lhdp \left( \Phi\left(g'\right)\lhdp \Phi(g_0)\right)  \nonumber\\
			=& \left[\left(\Phi(f)\lhdp \Phi\left({g'}^{(1)}\right)\right)\cdot\left(\Phi(f_0)\lhdp \Phi\left({g'}^{(2)}\right)\right)\right]\lhdp \Phi(g_0) -\Phi(ff_0)\lhdp\bigl(\Phi(g')\lhdp \Phi(g_0)\bigr) \label{eq:lem_proof}.
		\end{align}
		Moreover, using the lemma hypothesis~$\eqref{hyp0}$, we get 
		\begin{align*}
			\Phi(f)\lhdp \Phi\left({g'}^{(1)}\right)&=\Phi\left(f\lhd {g'}^{(1)}\right)\in\left\langle \ima(\phi) \right\rangle, \\
			\Phi(ff_0)\lhdp\left(\Phi\left(g'\right) \lhdp\Phi(g_0)\right)&=\Phi(ff_0)\lhdp\Phi(g'\lhd g_0),
		\end{align*}
		where $g'\lhd g_0\in \U{\h}^{\leq r}$ by lemma \ref{lem:filt+rhd}.
		 Thanks to hypothesis~\eqref{hyp1}, the first term of equation~\eqref{eq:lem_proof} is:
		\begin{align*}
			&\left[\left(\Phi(f)\lhdp \Phi\left({g'}^{(1)}\right)\right)\cdot\left(\Phi(f_0)\lhdp \Phi\left({g'}^{(2)}\right)\right)\right]\lhdp \Phi(g_0) \\
			=& \left[\left(\Phi(f)\lhdp \Phi\left({g'}^{(1)}\right)\right)\lhdp\Phi(g_0)\right]\cdot \left(\Phi(f_0)\lhdp\Phi\left({g'}^{(2)}\right)\right) \\
			+& \left(\Phi(f)\lhdp\Phi\left({g'}^{(1)}\right)\right)\cdot \left[\left(\Phi(f_0)\lhdp \Phi\left({g'}^{(2)}\right)\right)\lhdp\Phi(g_0)\right]).
		\end{align*}
		 So using the induction hypothesis, one has:
		\begin{align*}
			\Phi(ff_0)\lhdp\left(\Phi(g')\lhdp \Phi(g_0)\right)&=\left( \Phi(f)\lhdp \Phi\left(\left(g'\lhd g_0\right)^{(1)}\right)\right)\cdot \left( \Phi(f_0)\lhdp \Phi\left(\left(g'\lhd g_0\right)^{(2)}\right)\right). 
		\end{align*}
		By theorem~\ref{thm:extendUg}, we know $\Delta$ is a morphism for $\lhd $. As a consequence:
		\begin{align*}
			&\Phi(ff_0)\lhdp\left(\Phi(g')\lhdp \Phi(g_0)\right) \\
			=&\left( \Phi(f)\lhdp \Phi\left({g'}^{(1)}\lhd {g_0}^{(1)}\right)\right)\cdot \left( \Phi(f_0)\lhdp \Phi\left({g'}^{(2)}\lhd {g_0}^{(2)}\right)\right) \\
			=&\left[ \Phi(f)\lhdp
			\left(\Phi\left({g'}^{(1)}\right)\lhdp\Phi\left({g_0}^{(1)}\right)\right)\right]\cdot \left[ \Phi(f_0)\lhdp \left(\Phi\left({g'}^{(2)}\right)\lhdp\Phi\left({g_0}^{(2)}\right)\right)\right] \\
			=&\left[ \Phi(f)\lhdp
			\left(\Phi\left({g'}^{(1)}\right)\lhdp\Phi\left(g_0\right)\right)\right]\cdot \left( \Phi(f_0)\lhdp \Phi\left({g'}^{(2)}\right)\right) \\
			+&\left( \Phi(f)\lhdp
			\Phi\left({g'}^{(1)}\right)\right)\cdot \left[ \Phi(f_0)\lhdp \left(\Phi\left({g'}^{(2)}\right)\lhdp\Phi\left(g_0\right)\right)\right],
		\end{align*}
		where we used the hypothesis $\eqref{hyp0}$ (as $r<n$) for the second equality. Putting things together, we have:
		\begin{align*}
			&\Phi(ff_0)\lhdp\Phi(g'g_0) \\
			=& \color{blue}{\left[\left(\Phi(f)\lhdp \Phi\left({g'}^{(1)}\right)\right)\lhdp\Phi(g_0)\right]\cdot \left(\Phi(f_0)\lhdp\Phi\left({g'}^{(2)}\right)\right)} \\
			+& \color{red}{\left[\Phi(f)\lhdp\Phi\left({g'}^{(1)}\right)\right]\cdot \left(\left(\Phi(f_0)\lhdp \Phi\left({g'}^{(2)}\right)\right)\lhdp\Phi(g_0)\right)} \\
			-&\color{blue}{\left[ \Phi(f)\lhdp
				\left(\Phi\left({g'}^{(1)}\right)\lhdp\Phi\left(g_0\right)\right)\right]\cdot \left( \Phi(f_0)\lhdp \Phi\left({g'}^{(2)}\right)\right)} \\
			-&\color{red}{\left( \Phi(f)\lhdp
				\Phi\left({g'}^{(1)}\right)\right)\cdot \left[ \Phi(f_0)\lhdp \left(\Phi\left({g'}^{(2)}\right)\lhdp\Phi\left(g_0\right)\right)\right]}.
		\end{align*}
		Finally, as we supposed \eqref{hyp2}, the highlighted terms simplify:
		\begin{align*}
			&\Phi(ff_0)\lhdp \Phi(g'g_0) \\
			=&\color{blue}{\left[\Phi(f)\lhdp \left(\Phi\left({g'}^{(1)}\right)\cdot \Phi(g_0)\right) \right]\cdot  \left(\Phi(g_0)\lhdp \Phi\left({g'}^{(2)}\right)\right)} \\
			+&\color{red}{\left( \Phi(f)\lhdp \Phi\left({g'}^{(1)}\right) \right)\cdot \left[\Phi(f_0)\lhdp \left(\Phi\left({g'}^{(2)}\right)\cdot\Phi(g_0)\right)\right]}  \\
			=& \left( \Phi(f)\lhdp \Phi\left({g'}^{(1)}g_0\right)\right)\cdot  \left(\Phi(g_0)\lhdp \Phi\left({g'}^{(2)}\right)\right) \\
			+&\left( \Phi(f)\lhdp \Phi\left({g'}^{(1)}\right) \right)\cdot \left(\Phi(f_0)\lhdp \Phi\left({g'}^{(2)}g_0\right)\right).
		\end{align*}
		As $g_0\in\h$ and $\Delta$ is a morphism, one gets:
		\[
		\Delta(g)=\Delta(g'g_0)=\Delta(g')\Delta(g_0)={g'}^{(1)}\otimes {g'}^{(2)}g_0+{g'}^{(1)}g_0\otimes {g'}^{(2)}.
		\]
		As a consequence:
		\[
		\Phi(ff_0)\lhdp \Phi(g)=\left(\Phi(f')\lhdp\Phi\left(g^{(1)}\right)\right)\cdot\left(\Phi(f_0)\lhdp \Phi\left(g^{(2)}\right)\right).
		\]
		This proves the induction hypothesis for the rank $r+1$ and so the lemma. \qedhere
	\end{description}
\end{proof}

\begin{proof}[Proof of theorem \ref{thm:UnivPH}]
	The only thing to do is to prove equations~\eqref{PPL0}, \eqref{PPL1} and \eqref{PPL2} as the first part of the proposition comes from the universal property of enveloping algebras. To prove this result, it is enough to prove \eqref{PPL0}.
	Indeed,  let us suppose $\eqref{PPL0}$ is true. Let $f,g,h$ be three elements of $\U{\h}$. Then by construction of the product $\lhd$ over \U{\h}:
	\begin{align*}
		\left(\Phi(f)\cdot \Phi(g)\right)\lhdp \Phi(h)&=\Phi((fg)\lhd h) \\
		&=\Phi\left(\left(f\lhd h^{(1)}\right) \cdot \left(g\lhd h^{(2)}\right)\right) \\
		&=\Phi\left(f\lhd h^{(1)}\right)\cdot \Phi\left(g\lhd h^{(2)}\right) \\
		&=\left(\Phi\left(f\right)\lhdp \Phi\left(h^{(1)}\right)\right) \cdot \left( \Phi(g) \lhdp \Phi\left(h^{(2)}\right)\right).
	\end{align*}
Moreover:
\begin{align*}
	\left( \Phi(f)\lhdp  \Phi(g)\right) \lhdp \Phi(h)&=\Phi\left( (f\lhd g)\lhd h \right) \\
	&=\Phi\left( f\lhd \left(\left(g\lhd h^{(1)}\right)h^{(2)}\right)\right) \\
	&=\Phi(f)\lhdp \left( \left( \Phi(g)\lhdp\Phi\left(h^{(1)}\right)\right)\cdot\Phi\left(h^{(2)}\right)\right).
\end{align*}
Hence, properties~\eqref{PPL1} and~\eqref{PPL2} are implied by equation~\eqref{PPL0}.
To prove equation~\eqref{PPL0}, we proceed by an induction over $n=\max(\len(f),\len(g))$. By linearity of $\Phi$, it is enough to prove this for words over $\h$. 

\begin{description}
	\item[Initialization:] the case $n=0$ is straightforward. Then, consider the case $n=1$. By definition of $\lhdp$ extended to $A$ and the relations in theorem \ref{thm:extendUg} defining $\lhd $ with units in \U{\h}, we conclude equation~$\eqref{PPL0}$ is true whenever $f$ or $g$ is equal to $1$. If $f$ and $g$ are both elements of $\h$:
	\[
	\Phi(f\lhd g)=\phi(f\lhd g)=\phi(f)\lhdp \phi(g)=\Phi(f)\lhdp \Phi(g).
	\] 
	\item[Heredity:] assume there exists $n\geq 1$ such that for all $f,g\in\U{\h}^{\leq n}$. Then:
	\[
	\Phi(f\lhd g)=\Phi(f)\lhdp \Phi(g).
	\]
	Let us consider $(f,g)\in\U{\h}^{2}$ such that $\max(\len(f),\len(g))=n+1$.
	
	\textbf{Case 1:} $\len(f)\leq n$ and $\len(g)=n+1$.
	
	Then, up to terms with smaller degrees, there exists $g'\in\h^{\otimes n}$ and $g_0\in\h$ such that $g=g'g_0.$ From $\lhd$ construction of theorem~\ref{thm:extendUg}, lemma~\ref{lem:filt+rhd} and the induction hypothesis, one has:
	\begin{align*}
		\Phi(f\lhd g)&=\Phi\left(f\lhd\left(g'g_0\right)\right) \\
		&=\Phi\left( \left(f\lhd g'\right)\lhd g_0-f\lhd\left(g'\lhd g_0\right) \right) \\
		&= \left(\Phi(f)\lhdp \Phi(g')\right)\lhdp \Phi(g_0)-\Phi(f)\lhdp \left(\Phi(g')\lhdp \Phi(g_0)\right).
	\end{align*}
	Hypothesis~\eqref{hyp2} implies:
	\begin{equation*}
		\Phi(f\lhd g)=\Phi(f)\lhdp \left( \Phi\left(g'\right)\cdot \Phi(g_0)\right) =\Phi(f)\lhdp \Phi(g).
	\end{equation*}
	
	\textbf{Case 2:} $\len(f)=n+1$ and $\len(g)\leq n$.
	
	Hence, up to terms with smaller degrees, there exists $f_0\in\h$ and $f'\in\h^{\otimes n}$ such that $f=f'f_0.$ Then, by theorem \ref{thm:extendUg} and the induction hypothesis:
	\begin{align*}
		\Phi(f\lhd g)&=\Phi\left(\left(f'f_0\right)\lhd g\right) \\
		&=\Phi\left(\left(f'\lhd g^{(1)}\right)\left(f_0\lhd g^{(2)}\right)\right) \\
		&=\left(\Phi(f')\lhdp \Phi\left(g^{(1)}\right)\right)\cdot \left(\Phi(f_0)\lhdp \Phi\left(g^{(2)}\right)\right).
	\end{align*}
	Here, lemma \ref{lem:proofthm} applies as $\eqref{hyp0}$ is our induction hypothesis. As a consequence:
	\[
	\Phi(f'f_0)\lhdp \Phi(g)=\left(\Phi(f')\lhdp\Phi\left(g^{(1)}\right)\right)\cdot\left(\Phi(f_0)\lhdp \Phi\left(g^{(2)}\right)\right).
	\]
	So it proves $\Phi(f\lhd g)=\Phi(f)\lhdp \Phi(g)$ in this case.
	
	\textbf{Case 3:} $\len(f)=\len(g)=n+1.$
	
	Thus, there exist $g_0\in\h$ and $g'\in \h^{\otimes n}$ such that $g=g'g_0$. Consequently:
	\begin{align*}
		\Phi(f\lhd g)&=\Phi(f\lhd g'g_0) \\
		&= \Phi\bigl( \left(f\lhd g'\right)\lhd g_0-f\lhd\left(g'\lhd g_0\right)\bigr) \\
		&=\Phi\bigl((f\lhd g')\lhd g_0\bigr)-\Phi\bigl(f\lhd(g'\lhd g_0)\bigr).
	\end{align*} 
	However, by lemma \ref{lem:filt+rhd}, we know $f\lhd g'\in\U{\h}^{\leq {n+1}}$, $g'\lhd g_0\in\U{\h}^{\leq n}$ where $g_0\in\h$. As a consequence, from cases~1 and~2 of this proof, we get:
	\begin{align*}
		\Phi\bigl(\left(f\lhd g'\right)\lhd g_0\bigr)&=\bigl(\Phi(f)\lhdp \Phi(g')\bigr) \lhdp \Phi(g_0), \\
		\Phi\bigl(f\lhd\left(g'\lhd g_0\right)\bigr)&=\Phi(f)\lhdp \bigl(\Phi(g')\lhdp \Phi(g_0)\bigr).
	\end{align*}
	So, hypothesis~\eqref{hyp2} implies $\Phi(f\lhd g)=\Phi(f)\lhdp \Phi(g)$. Therefore we have proved equation~$\eqref{PPL0}$.
\end{description} 
Finally, the theorem is true by the induction principle.
\end{proof}

Then, we just give some classical tools useful to prove \CQMM{} theorem in the classic case~\cite{AlgHopf}.

\begin{Lemme}\label{lem:CQQM1}
	Let $H$ be a bialgebra and $n\in\N$. For all $(p_1,\dots,p_n)\in\Prim(H)^n$, we have:
	\[
	\tilde{\Delta}^{(n-1)}\left(p_1\cdot\ldots\cdot p_n\right)=\sum_{\sigma\in S_n} p_{\sigma(1)}\otimes \dots\otimes p_{\sigma(n)}.
	\]
		Moreover, if $k\geq n, \tilde{\Delta}^{(k)}\left(p_1\cdot\ldots\cdot p_n\right)=0$.
\end{Lemme}

\begin{Lemme}\label{lem:CQQM2}
	Let $H$ be a bialgebra and let $x\in H$ such that there exists $n\in\N$ with $\tilde{\Delta}^{(n)}(x)=0.$ Then:
	\[
	\tilde{\Delta}^{(n-1)}(x)\in\Prim(H)^{\otimes n}.
	\]
\end{Lemme}

\begin{Lemme}\label{lem:CQQM4}
	Let $H$ be a connected (as a coalgebra) bialgebra such that $I$ be a coideal such that $I\neq (0).$ Then, $I$ contains some primitive elements of $H$. 
\end{Lemme}

\subsection{The \CQMM{} theorem in the \PH{} case}
 
\begin{Cor}[\CQMM{} for \PH{} algebras]\label{Cor:CQQM}
	Suppose $\car(\K)=0$ and let $(H,\lhd)$ be a  cocommutative connected (as a coalgebra) \PH{} algebra. Then, $(H,\lhd)$ is isomorphic to the \PH{} algebra $\bigl(\U{\Prim(H)},\lhd\bigr)$ built in theorem \ref{thm:extendUg}.
\end{Cor}
This proof is quite similar to the classic \CQMM{} theorem where we use the coradical filtration. One may try to get the best version of it, appearing at theorem 4.5.1 in \cite{CartierPatras}, in the case of \PH{} algebras. 
\begin{proof}
	Let $(H,\lhd)$ be a cocommutative \PH{} algebra.
	First, we show that $H$ is generated by its primitive elements. Let $H'$ be the sub-bialgebra of $H$ generated by $\Prim(H)$. Let $x\in H\setminus\{0\}$.  As $H$ is connected as a coalgebra, for $k=\deg(x)$, we have $\tilde{\Delta}^{(k)}(x)=0.$
	We show by induction on $\deg(x)$ that $x\in H'$.
	
	\begin{description}
		\item[Initialization:] if $\deg(x)=0$, as $\tilde{\Delta}^{(0)}(x)=x-\e(x)1$ then $x\in\K\cdot 1$, so $x\in H'$. If $\deg(x)=1$ then $x$ is primitive, so $x\in H'$.
		\item[Heredity:] suppose there exists $k\geq 1$ such that $H^{\leq k} \subseteq H'$. Let $\deg(x)=k+1$, then by lemma~\ref{lem:CQQM2}, there exists $I$ a finite subset such that:
		\[
		\tilde{\Delta}^{(k)}(x)=\sum_{i\in I} x_1^{(i)}\otimes \dots \otimes x_{k+1}^{(i)},
		\] 
		where for all $m\in\IEM{1}{k+2}$ and $i\in I, x_m^{(i)}\in\Prim(H)$. As $H$ is cocommutative, the coproduct is invariant under the action of any element of the symmetric group. As a consequence, using $\car(\K)=0$:
		\[
		\tilde{\Delta}^{(k)}(x)=\frac{1}{(k+1)!}\sum_{\sigma\in S_{k+1}} \sum_{i\in I} x_{\sigma(1)}^{(i)}\otimes \dots \otimes x_{\sigma(k+1)}^{(i)}.
		\] 
		Thanks to lemma~\ref{lem:CQQM1}, we get:
		\[
		\tilde{\Delta}^{(k)}(x)=\tilde{\Delta}^{(k)}\left(\frac{1}{(k+1)!}\sum_{i\in I} x_1^{(i)}\cdot\ldots \cdot x_{k+1}^{(i)}\right).
		\]
		Then:
		\[
		\deg\left(x-\frac{1}{(k+1)!}\sum_{i\in I} x_1^{(i)}\cdot\ldots \cdot x_{k+1}^{(i)}\right)\leq k.
		\]
		By the induction hypothesis, $x\in H'$ because $\displaystyle\sum_{i\in I} x_1^{(i)}\cdot\ldots \cdot x_{k+1}^{(i)}\in H'.$ Consequently, this proves that $H$ is generated by its primitives elements.
	\end{description} 
	
	By theorem~\ref{thm:PrimPL}, $(\Prim(H),[,],\lhd)$ is a \PL{} algebra and consider $H$ as an associative algebra endowed with a linear map ${\lhd:H\otimes H \rightarrow H}.$
	We put ${\phi:\Prim(H)\rightarrow H}$ the injection of $\Prim(H)$ into $H$. The map $\phi$ satisfies for any $x,y\in\Prim(H)$:
	\begin{equation*}
		\phi([x,y])=xy-yx=\phi(x)\phi(y)-\phi(y)\phi(x) \text{ and }
		\phi(x\lhd y)=x\lhd y=\phi(x)\lhd \phi(y).
	\end{equation*}
	Moreover $\phi$ satisfies equations~$\eqref{hyp1}$ and $\eqref{hyp2}$ as any element in $\ima(\phi)$ is primitive in the \PH{} algebra $H$. So, applying theorem~\ref{thm:UnivPH} with $\lhdp \coloneqq \lhd$, there exists a unique algebra morphism ${\Phi:\U{\Prim(H)}\rightarrow H}$ satisfying equations  $\eqref{PPL0},\eqref{PPL1}, \eqref{PPL2}$ and $\Phi|_{\Prim(H)}=\phi$. However, $\Phi$ is also a coalgebra morphism as it sends primitive elements of $\U{\Prim(H)}$ on primitives elements of $H$ and for all $n\in\N, {\Phi\left(H^{\leq n}\right)\subseteq H^{\leq n}}$. Therefore, $\ima(\Phi)$ is a sub-\PH{} algebra of $H$ as equations~$\eqref{PPL1}$ and $\eqref{PPL2}$ become the relations~$\eqref{RPH1}$ and $\eqref{RPH2}$ in $\ima(\Phi)$. The existence of an inverse for the multiplication in the convolution algebra is a consequence of proposition~\ref{Prop:OptiLRPH} as $\ima(\Phi)$ is connected as a coalgebra.
	
	Note that $\ima(\Phi)$ contains the algebra generated by $\Prim(H)$. So $\ima(\Phi)=H$. Now, let us suppose by contradiction that $\Phi$ is not injective. Therefore, its kernel is a non-empty coideal in a connected bialgebra. Therefore by lemma \ref{lem:CQQM4}, $\ker(\Phi)$ contains primitive elements. So $\phi$ can not be injective. This gives a contradiction.
	
	Thus, $\Phi$ is an isomorphism of \PH{} algebras.	
\end{proof}
\begin{Rq}
In the particular case where $\Prim(H)$ is an abelian Lie algebra, one gets a version for Pre-Hopf algebras. As a consequence, we get theorem~$2.12$ from the work of J-M.~{Oudom} and D.~{Guin} \cite{OudomGuin} (with the additional hypothesis $\car(\K)=0$) with an isomorphism of Pre-Hopf algebra as $\Prim(S(\h))=\h$ for any Pre-Lie algebra $\h$.  
\end{Rq}

\subsection{A graded version}

\subsection{Reminders and notations}

Let us remind the definition of graded algebras and introduce some notations for the graded case.
\begin{defi}
	Let $(H,m,1_H,\Delta,\e)$ be a Hopf algebra with antipode $S$. It is said to be \emph{graded} if there exists a sequence $(H_n)_{n\in\N}$ of finite dimensional subspaces of $H$ such that:
	\begin{multicols}{2}
		\begin{enumerate}
			\item  $\displaystyle H = \bigoplus_{n\in\N} H_n$;\label{item:graded_vs}
			\item $\forall (m,n)\in\N^2, m\left(H_n\otimes H_m\right)\subseteq H_{n+m};$ \label{item:graded_alg1}
			\item 	$1_H\in H_0$;\label{item:graded_alg2}
			\item  $\displaystyle \forall n\in\N, \Delta\left(H_{n}\right)\subseteq \sum_{m=0}^n H_{n-m}\otimes H_{m}.$	\label{item:graded_big1}
			\item 	$\forall n\in\N\setminus\{0\}, \e(H_n)=(0);$\label{item:graded_big2}
			\item  $\forall n\in\N, S(H_n)\subseteq H_n$;\label{item:graded_Hopf}
		\end{enumerate}
	\end{multicols} 
	Moreover, $H$ is said \emph{connected} if $\dim(H_0)=1.$ An element $v\in H\setminus\{0\}$ in a graded space is said \emph{homogeneous} if there exists $n\in\N$ such that $v\in H_n$. If $v$ is homogeneous, we put $\deg(v)\coloneqq n$. Condition~\ref{item:graded_Hopf} is superfluous when $H$ is connected.
	\end{defi}
	\begin{Rq}
		A vector space $V$ is said \emph{graded}, if item~\ref{item:graded_vs} is satisfied. An algebra (or Lie algebra) $A$ is \emph{graded} if items~\ref{item:graded_alg1},\ref{item:graded_alg2} and \ref{item:graded_vs} are satisfied. A bialgebra is \emph{graded} if all items expected item~\ref{item:graded_Hopf} are satisfied. 
		The notions of connectedness is unchanged.
	\end{Rq}
	
	Let $V$ be a graded vector space and $\h$ be a graded Lie algebra.
	\begin{itemize}
		\item let $(\h,\libra{}{},\lhd)$ be a \PL{} algebra graded as a Lie algebra. It is called \emph{graded} if for any $n,m\in\N, \h_n \lhd \h_m \subseteq \h_{n+m}$;
		\item $T(V)$ is graded by $\displaystyle \bigoplus_{n\in\N} T(V)_n$ where $T(V)_n$ is the vector space generated by words such that the sum of the degrees of its letters is equal to $n$; 
		\item $\U{\h}$ inherits the grading of $T(\h)$ and they are both graded Hopf algebras;
		\item for any $k,l\in\N\setminus\{0\}$ we put $T(\h)^k_n\coloneqq \K\h^{\otimes k}\cap T(\h)_{n}.$ It defines a \emph{double grading} of $T(V)$. In particular, for any $k,l \in\N\setminus \{0\}$:
		
		\begin{minipage}{0.45\textwidth}
			\begin{itemize}
				\item $T(\h)^k_n$ is finite-dimensional;
				\item $T(\h)=\displaystyle \bigoplus_{k,n\in\N^2} T(\h)^k_n$;
			\end{itemize}
		\end{minipage}	
		\begin{minipage}{0.55\textwidth}
			\begin{itemize}
			\item $T(\h)^k_n\cdot T(\h)^l_m\subseteq T(\h)^{k+l}_{n+m}$;
			\item $\displaystyle\Delta\left(T(\h)^k_n\right)\subseteq \sum_{l=1}^k \sum_{m=1}^n T(\h)^l_m \otimes T(\h)^{k-l}_{n-m}$.
		\end{itemize}
		\end{minipage}
		\item for any $k,l\in\N\setminus\{0\}, \U{\h}^k_n\coloneqq \U{\h}^{\leq k}\cap \U{\h}_{n}$ defines a double grading of $\U{\h}$.
	\end{itemize}

\subsubsection{Preliminary lemmas and definitions}

\begin{defi}[graded \PH{} algebras] 
	Let $(H,\lhd)$ be a \PH{} algebra. We say it is \emph{graded} if $(H,m,1,\Delta,\e)$ is a graded bialgebra and for all $m,n\in\N,H_n\lhd H_m\subseteq H_{n+m}.$
\end{defi}

\begin{Lemme}\label{lem:Thgrade}
	Let $(\h,\libra{}{},\lhd)$ be a graded \PL{} algebra such that $\h_0=(0)$. Then:
	\[
	\forall (n,m,k)\in\N^3, T(\h)^k_n \lhd T(\h)_m\subseteq T(\h)^{k}_{n+m}.
	\]
\end{Lemme}
\begin{proof}
	In all this proof, we use intensively the construction from proposition~\ref{Prop:extendTg}.
We will prove by induction over $l$ the following property:
\[
\forall (n,m,k,l)\in\N^3, T(\h)_n^k\lhd T(\h)_m^l\subseteq T(\h)^k_{n+m}.
\]
\begin{description}
	\item[Initialization:] for $l=0$ it is easy as $T(\h)^{0}_0=(1)$. For $l=1$, one needs to show:
	\[
	\forall (n,m)\in\N^2,\forall k\in\N,T(\h)_n^k\lhd \h_m\subseteq T(\h)^k_{n+m}.
	\]
	 There, we use an induction over $k$. For $k=0$, it is obvious as $T(\h)^0=(1)$. For $k=1$, as $\h_0=(0)$,  for any $n\in\N,T(\h)^1_n=\h_n$. So, we need to prove:
	 \[
	 \forall (n,m)\in \N^2, \h_n\lhd \h_m\subseteq \h_{n+m}.
	 \] This is true as $\h$ is supposed to be a graded \PL{} algebra.
	 
	 Now, we suppose there exists $1\leq k$ such that for all $k'\leq k$, we have:
	 \begin{equation}
	 	\forall (n,m)\in\N^2,T(\h)_n^{k'}\lhd \h_m\subseteq T(\h)^{k'}_{n+m}. \tag{IH1} \label{IH1}
	 \end{equation}
	 Let $(n,m)\in\N^2$. We prove this result for $k+1$. Let $f\in T(\h)^{k+1}_n$ and $h\in\h_m$. Then, by linearity, we can suppose $f$ is a word. Hence, there exist $f'\in T(\h)^k_{n'}$ and $f_0\in \h_{n''}$ such that $f=f'f_0$ where $n=n'+n''$ with $1\leq n'$ and  $1\leq n''$. So, by construction:
	 \[
	 f\lhd h=(f'f_0)\lhd h=(f'\lhd h)f_0+f'(f_0\lhd h).
	 \]
	 By the induction hypothesis $\eqref{IH1}$, we have $f'\lhd h\in T(\h)^k_{n'+m}$ and $f_0\lhd h\in\h_{n''+m}$. Consequently:
	 \[
	 f\lhd h=\underbrace{(f'\lhd h)f_0}_{\in T(\h)^{k+1}_{n+m}}+\underbrace{f'(f_0\lhd h)}_{\in T(\h)^{k+1}_{n+m}}
	 \]
	 So the induction is initialized.
	\item[Heredity:] suppose there exists $1\leq l$ such that for all $l'\leq l$:
	\begin{equation}
		\forall (n,m)\in\N^2,\forall k\in\N, T(\h)_n^k\lhd T(\h)_m^{l'}\subseteq T(\h)^k_{n+m}. \tag{IH2} \label{IH2}
	\end{equation}
	We prove this result for $l+1$. Let $(n,m,k)\in\N^3,f\in T(\h)^k_{n}$ and $g\in T(\h)^{l+1}_m$. Then, by linearity, we can suppose $g$ is a word. Put $g=g'g_0$ where $g'\in T(\h)^l_{m'}$ and $g_0\in h_{m''}$ where $1\leq m', 1\leq m''$ and $m=m'+m''$. Consequently:
	\[
	f\lhd g=f\lhd (g'g_0)=(f\lhd g')\lhd g_0-f\lhd (g'\lhd g_0).
	\]
	By the induction hypothesis $\eqref{IH2}, f\lhd g'\in T(\h)^k_{n+m'}$ and $g'\lhd g_0\in T(\h)^l_{n+m''}$. So, using once more $\eqref{IH2}$:
	\[
	f\lhd g=\underbrace{(f\lhd g')\lhd g_0}_{\in T(\h)^{k}_{n+m}}-\underbrace{f\lhd (g'\lhd g_0)}_{\in T(\h)^k_{n+m}}. \qedhere
	\]
\end{description}
\end{proof}

As a consequence of lemma~\ref{lem:Thgrade}, we have
\begin{Prop}\label{Prop:T(V)PH}
	Let $(V,\lhd )$ be a graded magmatic algebra with $V_0=(0)$. Then, $(T(V),\cdot,\Delta, \lhd )$ is a cocommutative connected and graded right \PH{} algebra, using the extension of $\lhd $ in proposition~\ref{Prop:extendTg}.
\end{Prop}
\begin{proof}
	By proposition~\ref{Prop:extendTg}, we know that $\lhd $ is a coalgebra morphism and \eqref{RPH1} and \eqref{RPH2} hold. We grade $T(V)$ with $\left(T(V)_n\right)_{n\in\N}$. As $V_0=(0), \dim(T(V)_0)=1.$
	By lemma~\ref{lem:Thgrade}, all the hypotheses of proposition~\ref{Prop:graded=PH} are satisfied. So the proposition is true.
\end{proof}
\begin{Lemme}\label{lem:Uhgrade}
	Let $(\h,\libra{}{},\lhd)$ be a graded \PL{} algebra such that $\h_0=(0)$. Then:
	\[
	\forall (n,m,k)\in\N^3, \U{\h}^{\leq k}_n \lhd \U{\h}_m\subseteq \U{\h}^{\leq k}_{n+m}.
	\]
\end{Lemme}
\begin{proof}
	This lemma is a consequence of taking the quotient of $T(\h)$ by $I$. This ideal is homogeneous for the grading by the sum of the degrees of letters. But this not the case for the length of letters, which becomes a filtration. 
\end{proof}

The lemma \ref{lem:Uhgrade} and the corollary \ref{cor:UhPH} implies:
\begin{Cor}
	Let $(\h,\libra{}{},\lhd)$ be a graded \PL{} algebra. Then, $(\U{\h},\lhd)$ is a graded \PH{} algebra where $\lhd$ is the extension of theorem \ref{thm:extendUg} and the gradation is given by $\bigl(\U{\h}_n\bigr)_{n\in\N}$.
\end{Cor}

\subsubsection{The theorem}

We deduce a graded version of corollary \ref{Cor:CQQM}:

\begin{Cor}[\CQMM{} for graded \PH{} algebras]\label{Cor:gradCQQM}
	Let $(H,\lhd)$ be a cocommutative connected graded \PH{} algebra and assume ${\car(\K)=0}$. Then, $(H,\lhd)$ is isomorphic to the graded \PH{} algebra $\bigl(\U{\Prim(H)},\lhd\bigr)$ built in theorem~\ref{thm:extendUg}.
\end{Cor}
\begin{proof}
	The proof is the same as the one of corollary \ref{Cor:CQQM}. To adapt it to the graded case, one just needs to notice the map $\Phi$ from the proof of corollary~\ref{Cor:CQQM} is also a homogeneous morphism of degree $0$ for the gradations.
\end{proof}
\section{Applications}
\subsection{The sentence \PH{} algebra}

We introduce a tensor algebra over another tensor algebra. It has for basis the set of words over an alphabet of non-empty words that we will call a \emph{sentence}.
The space $T\left(T(V)_+\right)$ is the space of sentences over $V$. We will call \emph{sentence product} the concatenation product from $T\left(T(V)_+\right)$ denoted by $|$ and its deshuffle coproduct by $\Delta$. Its unit is the empty sentence $1$ and its counit is $0$ on every non-empty sentence of the basis.

\begin{defi}
	Let $V$ be vector space and consider the space $T(T(V)_+)$. We define the linear map ${\lhd: T(T(V)_+)\otimes T(T(V)_+)\rightarrow T(T(V)_+)}$ for any $S=S_1|\dots|S_n$ and $W=W_1|\dots|W_k$ by
	 \[
	 	S\lhd W=\sum_{f:\IEM{1}{k}\hookrightarrow \IEM{1}{n}} \left(S_1\cdot W_{f^{-1}(\{1\})}\right)|\dots |\left(S_n \cdot W_{f^{-1}(\{n\})}\right) 
	 	\text{ where } W_{\emptyset}=1.
	 \]
	 This space is graded by the number of words in a sentence. 
\end{defi}

\begin{Eg}
	With two sentences with only one word $\lhd$ is the concatenation product.
	We give below different examples of $\lhd$ product computation:
	\begin{align*}
		u \lhd (a|b)&=0, \\
		(ab|bc) \lhd uv&=abuv|bc + ab|bcuv, \\
		(ab|bc|de) \lhd (uv|w)&=abuv|bcw|de+abuv|bc|dew+ab|bcuv|dew \\
		 &+ abw|bcuv|de+abw|bc|deuv+ab|bcw|deuv.
	\end{align*}
\end{Eg}

We show it is indeed a \PH{} algebra structure
\begin{Prop}
	Let $V$ be a vector space. Then, $\left(T(T(V)),m,1,\Delta, \e, \lhd\right)$ is a \PH{} algebra.
\end{Prop}
\begin{Not}
	Given $n\in\N^*,(w_i)_{i\in\IEM{1}{n}}\in T(V)_+ ^ n$ and $w=w_1|\dots|w_k\in T(T(V)_+), J\subseteq \IEM{1}{k}$:
	\[
	\prod^|_{i\in I} w_i \coloneqq w_1|w_2|\ldots | w_n \text{ and } w|_J \coloneqq \prod^|_{j\in J} w_j.
	\]
\end{Not}
\begin{proof}
	For this purpose, thanks to proposition~\ref{Prop:graded=PH}, one has to show $\lhd$ is a coalgebra morphism and relations of definition~\ref{defi:RPH} are satisfied.
	
	First, we show $\lhd$ is a morphism of coalgebra. Let $S=S_1|\dots|S_n$ and $W=W_1|\dots|W_k$ be two sentences. Hence:
	\begin{align*}
		\Delta(S)\lhd \Delta(W) &= \left(\sum_{I\subseteq \IEM{1}{n}} S_I \otimes S_{I^c}\right) \lhd \left(\sum_{J\subseteq \IEM{1}{k}} W_J \otimes W_{J^c}\right) \\
								&= \sum_{I \subseteq \IEM{1}{n}} \sum_{J \subseteq \IEM{1}{k}} \sum_{\substack{f_1:J\hookrightarrow I \\ f_2:J^c\hookrightarrow I^c }} \prod_{i \in I}^{|} \left( S_i \cdot W_{f_1^{-1}(\{i\})} \right) \otimes \prod_{i \in I^c}^{|} \left( S_i \cdot W_{f_2^{-1}(\{i\})} \right) \\
								&= \sum_{I \subseteq \IEM{1}{n}} \sum_{f:\IEM{1}{k}\hookrightarrow \IEM{1}{n}} \left( S_i \cdot W_{f^{-1}(\{i\})} \right) \otimes \prod_{i \in I^c}^{|} \left( S_i \cdot W_{f^{-1}(\{i\})} \right) \\
								&= \sum_{f:\IEM{1}{k}\hookrightarrow \IEM{1}{n}} \Delta\left(S_1\cdot W_{f^{-1}(\{1\})}|\dots | S_n\cdot W_{f^{-1}(\{n\})}\right) \\
								&= \Delta(S\lhd W).
	\end{align*}
	
	Second, let us consider $T=T_1|\dots | T_m$. We show relations~\eqref{RPH1},\eqref{RPH2} from definition~\ref{defi:RPH}:
	\begin{align*}
		(S|T)\lhd W &= S_1|\dots | S_n| T_1 | \dots | T_m \lhd W_1 | \dots | W_k \\
					&=  \sum_{I \subseteq \IEM{1}{k}} \sum_{\substack{f_1: I \hookrightarrow \IEM{1}{n} \\ f_2: I^c \hookrightarrow \IEM{n+1}{n+m}}} \left.\left( \prod_{i \in \IEM{1}{k}}^| S_i \cdot W_{f_1^{-1}(\{i\})}\right) \middle| \left(\prod_{i \in \IEM{n+1}{n+m}}^| T_i\cdot W_{f_2^{-1}(\{i\})} \right)\right. \\
					&= \sum_{I \subseteq \IEM{1}{k}} \left(S \lhd W_I\right) | \left(T \lhd W_{I^c}\right) = \left.\left( S \lhd W^{(1)}\right) \middle| \left( T \lhd W^{(2)}\right)\right. .
	\end{align*}
	Thus, relation~\eqref{RPH1} is satisfied. Let us check property~\eqref{RPH2}:
	\begin{align*}
		(S\lhd T) \lhd W &= \left( \sum_{f:\IEM{1}{m}\hookrightarrow \IEM{1}{n}} \prod_{i \in \IEM{1}{n}}^| S_i\cdot T_{f^{-1}(\{i\})} \right) \lhd W \\
						&= \sum_{f:\IEM{1}{m}\hookrightarrow \IEM{1}{n}} \sum_{g:\IEM{1}{k}\hookrightarrow \IEM{1}{n}} \prod^|_{i \in \IEM{1}{n}} \left(S_i \cdot T_{f^{-1}(\{i\})} \cdot W_{g^{-1}(\{i\})}\right) \\
						&=  \sum_{\substack{f:\IEM{1}{m}\hookrightarrow \IEM{1}{n} \\ I\coloneqq \ima(f)}} \sum_{J \subseteq  \IEM{1}{k}} \sum_{\substack{g_1: J \hookrightarrow I \\ g_2: J^c \hookrightarrow I^c}} \prod^|_{i \in \IEM{1}{n}} U_i^{g_1,g_2}
	\end{align*}
	where $g_1=g|_{J}, g_2=g|_{J^c}$ and for $i\in I, U_i^{g_1,g_2}=S_i \cdot T_{f^{-1}(\{i\})} \cdot W_{g_1^{-1}(\{i\})}$ else $U_i^{g_1,g_2}=S_i \cdot W_{g_2^{-1}(\{i\})}$. Moreover, for any injection $g$ from $\IEM{1}{k}$ to $\IEM{1}{n}$ and $J\subseteq \IEM{1}{k}$ there exists a unique $h: J \hookrightarrow \IEM{1}{m}$ such that $g|_J = f \circ h$. Hence:
	\begin{align*}
		(S\lhd T) \lhd W &= S \lhd \left( \sum_{J \subseteq  \IEM{1}{k}} \sum_{h: J \hookrightarrow \IEM{1}{m}} \left.\left(\prod_{i \in \IEM{1}{m}}^| T_i \cdot W_{h^{-1}(\{i\})}\right) \middle| W_{J^{c}}\right. \right) \\
		&= S \lhd \left(\left.\left(T\lhd W^{(1)}\right)\middle| W^{(2)} \right.\right). \qedhere
	\end{align*}
\end{proof}

Corollary~\ref{Cor:CQQM} applies, as a consequence

Note that with the classic version, we are not able to conclude this isomorphism.

%

\subsection{A \PH{} algebra involving trees} \label{subsec:TreesPH}

\subsubsection{Trees and multisets}

\begin{defi}[Set of planar binary trees]
	We define $\Tree$ the set of \emph{planar rooted binary trees} with addition of one element $|$ which is an element with one leaf. We decompose this set with $\displaystyle\Tree=\bigcup_{n\in\N}\Tree_{n}$ where $\Tree_n$ is the subset of $\Tree$ containing trees with exactly $n+1$ leaves. Given $T$ a tree, we will denote $\Leaf(T)$ the set of leaves of this tree.
	
	Note that we will not represent explicitly the nodes of the trees of $\Tree$ as they are always at an intersection of two edges. 
	We denote $\K\Tree$ the $\K$-vector space generated by $\Tree$. It is a graded vector space with $\left( \K\Tree_{n}\right)_{n\in\N}$.
\end{defi}
\begin{Eg}
	For instance:
	\begin{align*}
		\Tree_0=\{ | \}, &&
		\Tree_1=\left\{\Y{}{}\right\}, &&
		\Tree_2=\left\{\balaisd{}{}{}, \balaisg{}{}{}\right\}.
	\end{align*}
\end{Eg}

\begin{defi}[Trees with decorated leaves]
	Let $X$ be a set. A \emph{binary tree with decorated leaves} is a couple $(T,f)$ where $T$ is an element of $\Tree$ and $f:\Leaf(T)\rightarrow X$. 
	We denote the set composed of those objects by $\Treedecleaf(X)$. For any $(T_1,T_2)\in\Treedecleaf(X)^2$, we denote $T_1\vee T_2$ the tree whose root has for left child $T_1$ and for right child $T_2$.
	
	A \emph{planar forest} is a word over the alphabet $\Treedecleaf(X)$. We denote the set of planar forests of such trees by $\Fdecleaf(X)$. For any binary forest ${F=T_1\dots T_k}$, $k$ is called the \emph{length} of $F$ and is denoted by $l(F)$. 
\end{defi}
\begin{Eg} 
	For instance with $X=\{a,b,c,d\}:$
	\[
	 \YY{$a$}{$d$}{$c$}{$b$}=\Y{$a$}{$d$} \vee \Y{$c$}{$b$} \text{ and } \balaisg{$a$}{$b$}{$c$}~\YY{$a$}{$d$}{$c$}{$b$}\neq \YY{$a$}{$d$}{$c$}{$b$}~\balaisg{$a$}{$b$}{$c$} \in\Fdecleaf(X).
	\]
\end{Eg}

\begin{defi}[Multiset] \label{defi:multiset}
	A \emph{multiset} $\mathcal{E}$ is a couple $(E,m)$ where $E$
	is a set that we will call the \emph{base set} of $\mathcal{E}$ and $m$ a map from $E$ to $\N$ called \emph{multiplicity}. We call $m^{-1}(\N\setminus\{0\})$ the \emph{support} of $\mathcal{E}$ and denote it by $\supp(\mathcal{E})$.
	We define the \emph{cardinality} $|\mathcal{E}|$ of $\mathcal{E}:=(E,m)$ by:
	\[
	|\mathcal{E}|:=\sum_{x\in E} m(x)=\sum_{x\in\supp(\mathcal{E})} m(x).
	\]
	We say a multiset $\mathcal{E}$ is finite if $|\mathcal{E}|$ is finite.
	Traditionally, multisets are denoted between double brackets and an element $x$ of $E$ appears as many times as $m(x)$. We do not mind of the order in which those elements appear.
\end{defi}
\begin{Egs}
	Here are some examples of multisets:
	\begin{itemize}
		\item We define the following multiset $(\IEM{1}{3},m)$ where $m(x)=x$ for all $x\in\IEM{1}{3}$. Then, this multiset can be represented by $\multiset{1,2,2,3,3,3} \text{ or by }\multiset{1,3,2,3,2,3}.$
		\item A set $E$ is identified to a multiset $\mathcal{E}=(E,m)$ where $m$ is a constant application equal to $1$. So, denoting $E=\{x_1,\dots,x_n\}$ where $n$ is a non-negative integer and $x_i\neq x_j$ for any $i\neq j$, we identify $\{x_1,\dots,x_n\}$ and $\multiset{ x_1,\dots,x_n}.$
	\end{itemize}
\end{Egs}	

\begin{defi}
	Let $E$ and $F$ be two sets, $f:E\rightarrow \N$ and $g:F\rightarrow \N$ two maps. We define:
	\begin{align*}
		&f\cup g:\left\lbrace\begin{array}{rcl}
			E \cup F & \rightarrow & \N, \\
			x & \mapsto& \begin{cases}
				f(x) &\text{ if } x\in E \text{ and } x\notin F, \\
				g(x) &\text{ if }x\in F \text{ and }x\notin E, \\
				f(x)+g(x) & \text{ if } x\in E\cap F.
			\end{cases}
		\end{array}\right.
	\end{align*}
	Let $\mathcal{E}=(E,m_E)$ and $\mathcal{F}=(F,m_F)$ be two multisets. We define the \emph{union} of these two multisets, denoted by $\mathcal{E}\cup\mathcal{F}$, as the multiset $(E\cup F, m_E\cup m_F).$
\end{defi}
\begin{Rq}
	Notice that $\supp\left(\mathcal{E} \cup \mathcal{F}\right)=\supp(\mathcal{E})\cup \supp(\mathcal{F}).$
\end{Rq}

\subsubsection{The algebraic structure}

\begin{defi}
	Let $S=S_1\dots S_n$ be an element of $\Fdecleaf(X)$ of length $n\in\N^*$ and $F=F_1\dots F_k\in\Fdecleaf(X)$ of length $k\in\N^*$. For each of these forests, we define by induction over $n$:
	\begin{align*}
		\text{If }n=1, F\curvearrowleft S&=\sum_{i=1}^k F_1\dots \left(F_i\vee S\right) \dots F_{k}, \\
		\text{otherwise, }F\curvearrowleft S&=(F\curvearrowleft S')\lhd S_n, \text{ where } S'= S_1\dots S_{n-1}.
	\end{align*}
\end{defi}

\begin{Eg}
	For instance, for $a,b,c,d,e,f,g,h,i\in X$:
	\begin{align*}
		&\Y{\footnotesize{$a$}}{\footnotesize{$b$}}\balaisd{\footnotesize{$c$}}{\footnotesize{$d$}}{\footnotesize{$e$}} \curvearrowleft \left(\balaisg{\footnotesize{$f$}}{\footnotesize{$g$}}{\footnotesize{$h$}} \Y{\footnotesize{$i$}}{\footnotesize{$j$}}\right) \\
		=&\left( \Y{\footnotesize{$a$}}{\footnotesize{$b$}}\balaisd{\footnotesize{$c$}}{\footnotesize{$d$}}{\footnotesize{$e$}}\curvearrowleft  \balaisg{\footnotesize{$f$}}{\footnotesize$g$}{\footnotesize{$h$}} \right)\lhd\Y{\footnotesize{$i$}}{\footnotesize{$j$}} \\
		=&\left(\raisebox{-0.3\height}{\begin{tikzpicture}[line cap=round,line join=round,>=triangle 45,x=0.225cm,y=0.225cm]
				\draw (0,0) -- (0,1);
				\draw (0,1) -- (-3,4);
				\draw (-2,3) -- (-1,4);
				\draw (0,1) -- (3,4);
				\draw (2,3) -- (1,4);
				\draw (1.5,3.5) -- (2,4);
				\draw[above] (-3,4) node {\footnotesize{$a$}};
				\draw[above] (-1,4) node {\footnotesize{$b$}};
				\draw[above] (1,4) node {\footnotesize{$f$}};
				\draw[above] (2,4) node {\footnotesize{$g$}};
				\draw[above] (3,4) node {\footnotesize{$h$}};
		\end{tikzpicture}}\,\balaisd{\footnotesize{$c$}}{\footnotesize{$d$}}{\footnotesize{$e$}}+\Y{\footnotesize{$a$}}{\footnotesize{$b$}} \,\raisebox{-0.3\height}{\begin{tikzpicture}[line cap=round,line join=round,>=triangle 45,x=0.225cm,y=0.225cm]
				\draw (0,0) -- (0,1);
				\draw (0,1) -- (-3,4);
				\draw (-2,3) -- (-1,4);
				\draw (-1.5,3.5) -- (-2,4);
				\draw (0,1) -- (3,4);
				\draw (2,3) -- (1,4);
				\draw (1.5,3.5) -- (2,4);
				\draw[above] (-3,4) node {\footnotesize{$c$}};
				\draw[above] (-2,4) node {\footnotesize{$d$}};
				\draw[above] (-1,4) node {\footnotesize{$e$}};
				\draw[above] (1,4) node {\footnotesize{$f$}};
				\draw[above] (2,4) node {\footnotesize{$g$}};
				\draw[above] (3,4) node {\footnotesize{$h$}};
		\end{tikzpicture}}\right) \curvearrowleft  \Y{\footnotesize{$i$}}{\footnotesize{$j$}} \\
		=&\raisebox{-0.3\height}{\begin{tikzpicture}[line cap=round,line join=round,>=triangle 45,x=0.225cm,y=0.225cm]
				\draw (0,0) -- (0,1);
				\draw (0,1) -- (-3,4);
				\draw (-2.5,3.5) -- (-2,4);
				\draw (-1,2) -- (1,4);
				\draw (0,3) -- (-1,4);
				\draw (-0.5, 3.5) -- (0,4);
				\draw (0,1) -- (3,4);
				\draw (2.5,3.5) -- (2,4);
				\draw[above] (-3,4) node {\footnotesize{$a$}};
				\draw[above] (-2,4) node {\footnotesize{$b$}};
				\draw[above] (-1,4) node {\footnotesize{$f$}};
				\draw[above] (0,4) node {\footnotesize{$g$}};
				\draw[above] (1,4) node {\footnotesize{$h$}};
				\draw[above] (2,4) node {\footnotesize{$i$}};
				\draw[above] (3,4) node {\footnotesize{$j$}};
		\end{tikzpicture}}\,\balaisd{\footnotesize{$c$}}{\footnotesize{$d$}}{\footnotesize{$e$}}+\raisebox{-0.3\height}{\begin{tikzpicture}[line cap=round,line join=round,>=triangle 45,x=0.225cm,y=0.225cm]
				\draw (0,0) -- (0,1);
				\draw (0,1) -- (-3,4);
				\draw (-2,3) -- (-1,4);
				\draw (0,1) -- (3,4);
				\draw (2,3) -- (1,4);
				\draw (1.5,3.5) -- (2,4);
				\draw[above] (-3,4) node {\footnotesize{$a$}};
				\draw[above] (-1,4) node {\footnotesize{$b$}};
				\draw[above] (1,4) node {\footnotesize{$f$}};
				\draw[above] (2,4) node {\footnotesize{$g$}};
				\draw[above] (3,4) node {\footnotesize{$h$}};
		\end{tikzpicture}}\,\raisebox{-0.3\height}{\begin{tikzpicture}[line cap=round,line join=round,>=triangle 45,x=0.225cm,y=0.225cm]
				\draw (0,0) -- (0,1);
				\draw (0,1) -- (-2,3);
				\draw (-1,2) -- (0,3);
				\draw (-0.5,2.5) -- (-1,3);
				\draw (0,1) -- (2,3);
				\draw (1.5,2.5) -- (1,3);
				\draw[above] (-2,3) node {\footnotesize{$c$}};
				\draw[above] (-1,3) node {\footnotesize{$d$}};
				\draw[above] (0,3) node {\footnotesize{$e$}};
				\draw[above] (1,3) node {\footnotesize{$i$}};
				\draw[above] (2,3) node {\footnotesize{$j$}};
		\end{tikzpicture}}+\YY{\footnotesize{$a$}}{\footnotesize{$b$}}{\footnotesize{$i$}}{\footnotesize{$j$}} \,\raisebox{-0.3\height}{\begin{tikzpicture}[line cap=round,line join=round,>=triangle 45,x=0.225cm,y=0.225cm]
				\draw (0,0) -- (0,1);
				\draw (0,1) -- (-3,4);
				\draw (-2,3) -- (-1,4);
				\draw (-1.5,3.5) -- (-2,4);
				\draw (0,1) -- (3,4);
				\draw (2,3) -- (1,4);
				\draw (1.5,3.5) -- (2,4);
				\draw[above] (-3,4) node {\footnotesize{$c$}};
				\draw[above] (-2,4) node {\footnotesize{$d$}};
				\draw[above] (-1,4) node {\footnotesize{$e$}};
				\draw[above] (1,4) node {\footnotesize{$f$}};
				\draw[above] (2,4) node {\footnotesize{$g$}};
				\draw[above] (3,4) node {\footnotesize{$h$}};
		\end{tikzpicture}}+\Y{\footnotesize{$a$}}{\footnotesize{$b$}}\,\raisebox{-0.3\height}{\begin{tikzpicture}[line cap=round,line join=round,>=triangle 45,x=0.225cm,y=0.225cm]
				\draw (0,0) -- (0,1);
				\draw (0,1) -- (-4,5);
				\draw (-3,4) -- (-2,5);
				\draw (-2.5,4.5) -- (-3,5);
				\draw (0,1) -- (4,5);
				\draw (3, 4) -- (2,5);
				\draw (-1.5,2.5) -- (1,5);
				\draw (0,4) -- (-1,5);
				\draw (-0.5,4.5)--(0,5);
				\draw[above] (-4,5) node {\footnotesize{$c$}};
				\draw[above] (-3,5) node {\footnotesize{$d$}};
				\draw[above] (-2,5) node {\footnotesize{$e$}};
				\draw[above] (-1,5) node {\footnotesize{$f$}};
				\draw[above] (0,5) node {\footnotesize{$g$}};
				\draw[above] (1,5) node {\footnotesize{$h$}};
				\draw[above] (2,5) node {\footnotesize{$i$}};
				\draw[above] (4,5) node {\footnotesize{$j$}};
			\end{tikzpicture}
		}.
	\end{align*}
\end{Eg}

\begin{defi}\label{defi:Fforet}
	We define a multiset of elements of $\Fdecleaf(X)$, for any trees $T_1,\dots,T_k,T_{k+1}$ by:
	\begin{align*} 
		\F\left(T_1\right)&=\multiset{T_1}, \\
		\F\left(T_1\dots T_k T_{k+1}\right)&=\bigcup_{S\in\F(T_1\dots T_k)} \bigcup_{r=1}^{m_{\F\left(T_1\dots T_k\right)}(S)}\multiset{ST_{k+1}} \bigcup \bigcup_{i=1}^k\F\bigl(T_1\dots \left(T_i\vee T_{k+1}\right)\dots T_k\bigr).
	\end{align*}  
\end{defi}
\begin{Rq}
	The term $\displaystyle \bigcup_{r=1}^{m_{\F\left(T_1\dots T_k\right)}(S)}\multiset{ST_{k+1}}$ is the multiset with one element $\multiset{ST_{k+1}}$ appearing as many times as occurrences of $S$ in $\F\left(T_1\dots T_k\right).$  
\end{Rq}
\begin{Eg}
	For instance:
	\begin{align*}
		\F\left( T_1 T_2 \right)&=\multiset{T_1T_2,  \Y{$T_1$}{$T_2$}}, \\
		\F\left( T_1 T_2T_3\right)&=\multiset{T_1T_2T_3,\Y{$T_1$}{$T_2$}T_3,T_1\Y{$T_2$}{$T_3$},  \Y{$T_1$}{$T_3$}T_2 , \balaisg{$T_1$}{$T_3$}{$T_2$},\balaisd{$T_1$}{$T_2$}{$T_3$}}.
	\end{align*}
\end{Eg}
\begin{Rq}
	Having a multiset here is really important. Otherwise, in the case where $T_3=T_2$, one element is left in $\F\left( T_1 T_2T_3\right)$.
\end{Rq}

\begin{defi}
	Let $X$ be a set.
	We consider the Hopf algebra $(\K\FLT(X),m,1,\Delta,\e)$ where $m$ is the concatenation of forests, $1$ is the empty forest, $\Delta$ is the deshuffle coproduct of forests and $\e$ is the linear map equal to one for the empty forest and it is zero otherwise.
	We define:
	\[
	F\lhd F'=\sum_{S\in\F(F')} (-1)^{l(S)+l(F')} m_{\F(F')}(S)\,F\curvearrowleft S.
	\]
	According to remark~\ref{Rq:Card}, this sum can be rewritten over a sum over a symmetric group.
\end{defi}
\begin{Rq}\label{Rq:Card}
	By definition of those multisets from definition~\ref{defi:Fforet}, one has for all $n\in\N$:
	\[
	\forall \, T_1,\dots,T_n \in \Treedecleaf(X), \left|\F(T_1\dots T_n)\right|=n!.
	\]
	Hence, there exists a multiset surjection between this set and the $n$-th symmetric group $S_n$ enabling to index the sum over symmetric groups. We will not detail it in this paper.
\end{Rq}

\begin{Eg}
	We give some examples of computations in this algebra with simple trees decorated by a set of cardinality one. So we do not need to represent decorations here:
	\begin{align*}
		\Y{}{} \lhd \Y{}{}\Y{}{} &= \begin{tikzpicture}[anchor = baseline, baseline = 1.5em, x=1em, y=1em]
			\draw (0,0) -- (0,1); 
			\draw (0,1) -- (-2,3);
			\draw (0,1) -- (2,3);
			\draw (1.5,2.5) -- (1,3);
			\draw (-1.5,2.5) -- (-1,3);
			\draw (-0.75,1.75) -- (0.5,3);
			\draw (0,2.5) -- (-0.5,3);
		\end{tikzpicture} - \begin{tikzpicture}[anchor = baseline, baseline = 1.5em, x=1em, y=1em]
		\draw (0,0) -- (0,1); 
		\draw (0,1) -- (-2,3);
		\draw (0,1) -- (2,3);
		\draw (1.5,2.5) -- (1,3);
		\draw (-1.5,2.5) -- (-1,3);
		\draw (0.75,1.75) -- (-0.5,3);
		\draw (0,2.5) -- (0.5,3);
		\end{tikzpicture}, \\ 
		\Y{}{}\Y{}{} \lhd \Y{}{} &= \YY{}{}{}{}\Y{}{} + \Y{}{}\YY{}{}{}{},\\
		\Y{}{} \lhd \Y{}{}\Y{}{}\Y{}{} &= \sum_{S\in\F\left(\Yind{}{}\Yind{}{}\Yind{}{}\right)} (-1)^{l(S)+3} m_{\F\left(\Yind{}{}\Yind{}{}\Yind{}{}\right)} \Y{}{}\curvearrowleft S \\
		&=\begin{tikzpicture}[anchor = baseline, baseline = 1.5em, x=1em, y=1em]
			\draw (0,0) -- (0,1);
			\draw (0,1) -- (2.5,3.5);
			\draw (0,1) -- (-2.5,3.5);
			\draw (0.75,1.75) -- (-1.,3.5);
			\draw (-0.5,3.) -- (0,3.5);
			\draw (1.375,2.375) -- (0.25,3.5);
			\draw (0.75,3) -- (1.25,3.5);
			\draw (2,3) -- (1.5,3.5);
			\draw (-2.,3) -- (-1.5,3.5);
		\end{tikzpicture} + \begin{tikzpicture}[anchor = baseline, baseline = 1.5em, x=1em, y=1em]
		\draw (0,0) -- (0,1);
		\draw (0,1) -- (2.5,3.5);
		\draw (0,1) -- (-2.5,3.5);
		\draw (0.75,1.75) -- (-1.,3.5);
		\draw (-0.5,3.) -- (0,3.5);
		\draw (0.75,3) -- (0.25,3.5);
		\draw (0.125,2.375) -- (1.25,3.5);
		\draw (2,3) -- (1.5,3.5);
		\draw (-2.,3) -- (-1.5,3.5);
		\end{tikzpicture} - \begin{tikzpicture}[anchor = baseline, baseline = 1.5em, x=1em, y=1em]
		\draw (0,0) -- (0,1);
		\draw (0,1) -- (2.5,3.5);
		\draw (0,1) -- (-2.5,3.5);
		\draw (-0.5,3) -- (-1.,3.5);
		\draw (-1.25,2.25) -- (0,3.5);
		\draw (1.375,2.375) -- (0.25,3.5);
		\draw (0.75,3) -- (1.25,3.5);
		\draw (2,3) -- (1.5,3.5);
		\draw (-2.,3) -- (-1.5,3.5);
		\end{tikzpicture} - 2~\begin{tikzpicture}[anchor = baseline, baseline = 1.5em, x=1em, y=1em]
		\draw (0,0) -- (0,1);
		\draw (0,1) -- (2.5,3.5);
		\draw (0,1) -- (-2.5,3.5);
		\draw (0.125,2.375) -- (-1.,3.5);
		\draw (-0.5,3.) -- (0,3.5);
		\draw (0.75,3) -- (0.25,3.5);
		\draw (-0.625,1.625) -- (1.25,3.5);
		\draw (2,3) -- (1.5,3.5);
		\draw (-2.,3) -- (-1.5,3.5);
		\end{tikzpicture} + \begin{tikzpicture}[anchor = baseline, baseline = 1.5em, x=1em, y=1em]
		\draw (0,0) -- (0,1);
		\draw (0,1) -- (2.5,3.5);
		\draw (0,1) -- (-2.5,3.5);
		\draw (-0.5,3) -- (-1.,3.5);
		\draw (-1.25,2.25) -- (0,3.5);
		\draw (0.75,3) -- (0.25,3.5);
		\draw (-0.625,1.625) -- (1.25,3.5);
		\draw (2,3) -- (1.5,3.5);
		\draw (-2.,3) -- (-1.5,3.5);
		\end{tikzpicture} \\
		&=\left( \Y{}{} \lhd \Y{}{}\Y{}{}\right) \lhd \Y{}{} - \Y{}{}\lhd \left( \Y{}{}\Y{}{} \lhd \Y{}{}\right).
	\end{align*}
\end{Eg}

\begin{Prop}
	Let $X$ be a set. The algebra $(\K\FLT(X),m,1,\Delta,\e, \lhd)$ is a \PH{} algebra.
\end{Prop}
\begin{proof}
	Let $X$ be a set of decorations. Let us check it is indeed a \PH{} algebra
	\begin{itemize}
		\item given two forests $F,F'$ and a tree $T$, the construction of $\F$ from definition~\ref{defi:Fforet} implies:
		\begin{align}
			& (-1)^{l(F')}F \lhd  (F'T) \notag \\
			=& (-1)^{l(F')} \sum_{S\in\F\left(F'T\right)}  (-1)^{l(S) + l(F'T)} m_{\F(F'T)}(S) F \curvearrowleft S \notag \\
						  =& \sum_{S\in\F\left(F'T\right)}  (-1)^{l(S) +1} m_{\F(F'T)}(S) F \curvearrowleft S \notag \\
						  =& \sum_{S\in\F\left(F'\right)}  (-1)^{l(ST)} m_{\F(F')}(S) F \curvearrowleft (ST) + \sum_{S\in\F\left(F'\right)}  (-1)^{l(S) +1} m_{\F(F'T)}(S) F \curvearrowleft (S \lhd T) \notag \\
						  =& \sum_{S\in\F\left(F'\right)}  (-1)^{l(ST) } m_{\F(F')}(S) (F \curvearrowleft S) \lhd T + \sum_{S\in\F\left(F'\lhd T\right)}  (-1)^{l(S) +1} m_{\F(F')}(S) F \curvearrowleft S \notag \\
						  =& (-1)^{l(F')} \left[ (F \lhd F') \lhd T - F \lhd (F' \lhd T)\right].  \label{eq:stat1}
		\end{align}
		\item given three forests $F,G,H$, we prove $FG\lhd H = \left(F \lhd H^{(1)}\right)\cdot \left( G \lhd H^{(2)} \right)$ by induction over $l(H).$
		The initialization is straightforward. Then, suppose there exists $n\in\N^*$ such that for all forests $H$ with $l(H)\leq n$, this statement is true. Let $T$ be a tree and put $H'= H T$ with $l(H)=n$. Then, according to definition~\ref{defi:Fforet} and the induction hypothesis:
		\begin{align*}
			&(-1)^{l(H)}FG \lhd HT \\
			=& \sum_{S\in\F\left(HT\right)}  (-1)^{l(S) } m_{\F(HT)}(S) (FG) \curvearrowleft S \\
			=&\sum_{S'\in\F\left(H\right)}  (-1)^{l(S')} m_{\F(H)}(S) (FG) \curvearrowleft (S'T)  + (-1)^{l(S'\lhd T)+1} m_{\F(H)}(S') (FG) \curvearrowleft (S'\lhd T) \\
			=&\sum_{S'\in\F\left(H\right)}  (-1)^{l(S')} m_{\F(H)}(S) (FG \curvearrowleft S')\lhd T  - (-1)^{l(S'\lhd T)} m_{\F(H)}(S') (FG) \curvearrowleft (S'\lhd T) \\
			=& (-1)^{l(H)}\left( \left( F \lhd H^{(1)}\right)\cdot \left( G \lhd H^{(2)}\right)\right)\lhd T - \left( F \lhd \left( H \lhd T \right)^{(1)}\right)\cdot \left( G \lhd \left( H \lhd T\right)^{(2)}\right).
		\end{align*}
		By definition of the deshuffle coproduct and the result~\eqref{eq:stat1}, we end with:
		\begin{align*}
			&FG \lhd HT \\
			 &=\left( \left( F \lhd H^{(1)}\right)\cdot \left( G \lhd H^{(2)}\right)\right)\lhd T - \left( F \lhd \left( H^{(1)} \lhd T \right)\right)\cdot \left( G \lhd H^{(2)} \right) \\
			 &- \left( F \lhd  H^{(1)}\right) \cdot \left( G \lhd \left( H^{(2)} \lhd T\right)\right) \\
			 &=\left( \left( F \lhd H^{(1)}\right)\cdot \left( G \lhd H^{(2)}\right)\right)\lhd T + F \lhd \left( H^{(1)}T\right) \cdot \left(G \lhd H^{(2)}\right) \\
			 & -  \left(\left( F \lhd H^{(1)}\right)\lhd T\right)\cdot \left(G \lhd H^{(2)}\right)+ \left(F \lhd H^{(1)}\right) \cdot \left(G \lhd
			  H^{(2)} T\right) \\  &- \left(F \lhd H^{(1)}\right)\cdot  \left(\left( G \lhd H^{(2)}\right)\lhd T \right) \\
			 &= \left( F \lhd \left(HT\right)^{(1)} \right) \cdot \left(G \lhd (HT)^{(2)}\right).
		\end{align*}
		\item Finally, let us show $\Delta(F' \lhd G)= \Delta(F')\lhd \Delta(G)$ by induction over $l(F')+ l(G)$. The initial case with $l(F')+l(G)=2$ is straightforward. For the heredity, let us there exists $n>2$ such that for any $F',G$ with $l(F')+l(G)\geq 2, \Delta(F' \lhd G)= \Delta(F')\lhd \Delta(G)$. In the case where $l(F')>2$, then there exists $F,H\in\Treedecleaf(X)$ with such that $F'= FH$. Hence, using the second point, the induction hypothesis and the cocommutativity of $\Delta$:
		\begin{align*}
			\Delta(FH \lhd G) &= \Delta\left(F \lhd G^{(1)}\right)\cdot \Delta\left( H \lhd G^{(2)}\right) \\
								&= \left(F^{(1)} \lhd G^{(1)}\right) \cdot \left( H^{(1)} \lhd G^{(3)} \right) \otimes \left(F^{(2)} \lhd G^{(2)}\right) \cdot \left( H^{(2)} \lhd G^{(4)} \right) \\
								&= \left(F^{(1)} \lhd G^{(1)}\right) \cdot \left( H^{(1)} \lhd G^{(2)} \right) \otimes \left(F^{(2)} \lhd G^{(3)}\right) \cdot \left( H^{(2)} \lhd G^{(4)} \right) \\
								&= \left( FH \right)^{(1)} \lhd G^{(1)} \otimes \left( FH \right)^{(2)} \lhd G^{(2)} \\
								&= \Delta(FH) \lhd \Delta(G).
		\end{align*}
		In the case where $l(F')=1$, then there exist a forest $F$ and a tree $T$ such that $G= FT$. So using equation~\eqref{eq:stat1} and our induction hypothesis:
		\begin{align*}
			\Delta(F'\lhd G) &= \Delta\left( F' \lhd (F \lhd T)\right) - \Delta((F' \lhd F) \lhd T) \\
							&= \Delta \left( F'\right) \lhd \Delta (F \lhd T) - \Delta(F' \lhd F) \lhd \Delta(T) \\
							&=\left( F'^{(1)} \lhd F^{(1)}\right) \lhd T^{(1)} \otimes \left( F'^{(2)} \lhd F^{(2)}\right)\lhd T^{(2)} \\ &-  F'^{(1)} \lhd \left(F^{(1)} \lhd T^{(1)}\right) \otimes  F'^{(2)} \lhd \left(F^{(2)}\lhd T^{(2)}\right) \\
							&= F'^{(1)} \lhd (FT) ^{(1)} \otimes F'^{(2)} \lhd (FT)^{(2)} \\
							&= \Delta(F')\lhd \Delta(G).
		\end{align*}
	\end{itemize}
	Therefore, it shows $(\K\FLT(X),m,1,\Delta,\e, \lhd)$ is a \PH{} algebra by propositioné~\ref{Prop:OptiLRPH}.
\end{proof}

As a consequence, we can apply corollary~\ref{Cor:CQQM} and deduce that $\K\FLT(X)\approx \U{\Treedecleaf(X)}$ where $\Treedecleaf(X)$ is seen as a \PL{} algebra whose Lie bracket is the commutator for the concatenation and $\lhd$ is the grafting operator. This operator for trees is magmatic.

Hence, it turns out $\K\FLT(V)$ where $V$ is one dimensional is isomorphic to the Munthe-Kaas--Wright algebra $H_{MKW}$~\cite{Munthe_Kaas_2006} as \PH{} algebras. One can give the explicit isomorphism but it is not relevant to detail it here.

\subsection{Conclusion}

We made a link between the notions of left/right \PL{} algebras and under some hypotheses, we linked left and right \PH{} algebras. We proved a \CQMM{} theorem version for right \PH{} algebras. It enables us to state complex isomorphisms between \PH{} algebras. It is a step forward for the study of the algebraic structure of \PH{} algebra.
Moreover, one can investigate the following questions:
\begin{enumerate}
	\item Can we recover the stronger version involving group-like elements ~\cite[theorem~4.5.1]{CartierPatras} in the \PH{} case ? This requires to look at Post-groups structures~\cite{Al_Kaabi_2022}. 
	\item Is there a way to get an analogous of proposition~\ref{Prop:OptiLRPH} without any hypothesis? This is particularly connected to the question of existence of a non-cocommutative \PH{} algebra.
	\item In this paper, we did not talk about the subadjacent Hopf algebra hidden in a \PH{} algebra like at theorem 1 from \cite{siso} or at theorem 2.4 from \cite{LPostLie}. Is there a way to improve our results using this notion? 
\end{enumerate}

\bibliography{biblio_PH_V10}
\bibliographystyle{plain}

\end{document}